\documentclass[
11pt, 
oneside, 
english, 
singlespacing, 
nohyperref, 
headsepline, 
]{article} 

\usepackage{changes}
\usepackage{ulem}
\usepackage{comment}
\makeatletter
\@namedef{Changes@AuthorColor}{blue}
\colorlet{Changes@Color}{blue}
\makeatother
\usepackage{geometry}
\usepackage[utf8]{inputenc} 
\usepackage[T1]{fontenc} 
\usepackage{lmodern}
\usepackage{amsmath,amsthm,amssymb,stmaryrd,mathrsfs,dsfont,cases}
\usepackage{graphicx}
\usepackage{xcolor}
\definecolor{Green}{RGB}{9,100,40}
\definecolor{Red}{RGB}{150,0,24}
\usepackage{subfigure}
\usepackage{enumitem}
\usepackage{float}
\usepackage[colorlinks=true, citecolor=Green, urlcolor=Red, linkcolor=Red]{hyperref}
\usepackage{chngcntr}

\usepackage[autostyle=true]{csquotes} 

\newcommand\N{\mathbb{N}}

\newcommand\D{\mathbb{D}}

\newcommand\R{\mathbb{R}}

\newcommand\Ub{\mathbb{U}}
\newcommand\Pb{\mathbb{P}}
\newcommand\Eb{\mathbb{E}}
\newcommand\Hb{\mathbb{H}}

\newcommand\F{\mathscr{F}}
\newcommand\Gscr{\mathscr{G}}
\newcommand\n{\mathfrak{n}}
\newcommand\tb{t^{\bullet}}
\newcommand\zb{z^{\bullet}}

\newcommand\Gcal{\mathcal{G}}

\newcommand\Xcal{\mathcal{X}}
\newcommand\Mcal{\mathcal{M}}
\newcommand\Dcal{\mathcal{D}}

\newcommand\Rcal{\mathcal{R}}

\newcommand{\Pbb}[3]{P_{#1}^{#2\rightarrow #3}}
\newcommand{\Ebb}[3]{E_{#1}^{#2\rightarrow #3}}

\theoremstyle{plain}
\newtheorem{Thm}{Theorem}[section]
\newtheorem*{Thm*}{Theorem}
\newtheorem{Prop}[Thm]{Proposition}
\newtheorem*{Prop*}{Proposition}
\renewenvironment{proof}{{\bfseries Proof.}}{\qed}
\newtheorem{Def}[Thm]{Definition}
\newtheorem*{Def*}{Definition}
\newtheorem{Lem}[Thm]{Lemma}
\newtheorem*{Cor*}{Corollary}
\newtheorem{Cor}[Thm]{Corollary}

\theoremstyle{definition}

\newtheorem*{Ex*}{Example}

\newtheorem*{Rk}{Remark}

\begin{document}

\vglue30pt

\centerline{\large\bf  Growth-fragmentation process embedded in a planar Brownian excursion  }

\bigskip
\bigskip

 \centerline{by}

\medskip

 \centerline{Elie A\"{i}d\'ekon\footnote{\scriptsize LPSM, Sorbonne Universit\'e Paris VI, and Institut Universitaire de France, {\tt elie.aidekon@upmc.fr}} and William Da Silva\footnote{\scriptsize LPSM, Sorbonne Universit\'e Paris VI, {\tt william.da-silva@lpsm.paris}}}

\bigskip

\bigskip

{\leftskip=2truecm \rightskip=2truecm \baselineskip=15pt \small

\noindent{\slshape\bfseries Summary.}

\noindent The aim of this paper is to present a self-similar growth-fragmentation process linked to a Brownian excursion in the upper half-plane $\Hb$, obtained by cutting the excursion at horizontal levels.  We prove that the associated growth-fragmentation is related to one of the growth-fragmentation processes introduced by Bertoin, Budd, Curien and Kortchemski in \cite{BBCK}.

\medskip


\medskip
 

} 

\bigskip
\bigskip

\section{Introduction}

We consider a Brownian excursion in the upper half-plane $\Hb$ from $0$ to a positive real number $z_0$. For $a>0$, if the excursion hits the set $\{z\in \mathbb{C}\,:\, \Im(z)=a\}$ of points with imaginary part $a$, it will make a countable number of excursions above it, that we denote by $(e_i^{a,+}, \, i \ge 1)$. For any such excursion, we let $\Delta e_i^{a,+}$ be the difference between the endpoint of the excursion and its starting point, which we will refer to as the {\it size} or {\it length} of the excursion. Since both points have the same imaginary part, the collection $(\Delta e_i^{a,+}, \, i\ge 1)$ is a collection of real numbers and we suppose that they are ranked in decreasing order of their magnitude.  Our main theorem describes the law of the process $(\Delta e_i^{a,+}, \, i \ge 1)_{a\ge 0}$ indexed by $a$ in terms of a self-similar growth-fragmentation. We refer to \cite{B} and \cite{BBCK} for background on growth-fragmentations. Let us describe the growth-fragmentation process involved in our case. \\

Let $Z=(Z_a)_{0\le a<\zeta}$ be the positive self-similar Markov process of index $1$ whose Lamperti representation is
\[
Z_a = z_0\exp(\xi(\tau(z_0^{-1}a))), 
\]
where $\xi$ is the L\'evy process with Laplace exponent
\begin{equation} \label{Lapl loc largest}
    \Psi(q) = -\frac4\pi q + \frac2\pi \int_{y>-\ln(2)} \left(\mathrm{e}^{q y}-1-q(\mathrm{e}^y-1)\right) \frac{\mathrm{e}^{-y}\mathrm{d}y}{(\mathrm{e}^y-1)^2}, \quad q<3,
\end{equation}
$\tau$ is the time change 
\[\tau(a) = \inf\left\{s\ge 0, \; \int_{0}^s \mathrm{e}^{\xi(u)}\mathrm{d}u > a\right\},\]
and $\zeta = \inf\{a\ge 0, \; Z_a=0\}$. The \emph{cell system} driven by $Z$ can be roughly constructed as follows. The size of the so-called \emph{Eve cell} is $z_0$ at time $0$ and evolves according to $Z$. Then, conditionally on $Z$, we start at times $a$ when a jump $\Delta Z_a = Z_a-Z_{a-}$ occurs independent processes  starting from $-\Delta Z_a$, distributed as $Z$ when $\Delta Z_a<0$ and as $-Z$ when $\Delta Z_a>0$. These processes represent the sizes of the daughters of the Eve particle. Then repeat the process for all the daughter cells:  at each jump time of the cell process, start an independent copy of the process $Z$ if the jump is negative, $-Z$ if the jump is positive, with initial value the negative of the corresponding jump. This defines the sizes of  the cells of the next generation and we proceed likewise. We then define, for $a\ge 0$, $\overline{\mathbf{X}}(a)$ as the collection of sizes of cells alive at time $a$, ranked in decreasing order of their magnitude.

Growth-fragmentation processes were introduced in \cite{B}. Beware that the growth-fragmentation process we just defined is not included in the framework of \cite{B} or \cite{BBCK} because we allow cells to be created at times corresponding to positive jumps, giving birth to cells with negative size. Therefore, the process $\overline{\mathbf{X}}$ is not a true growth-fragmentation process. The formal construction of the process $\overline{\mathbf{X}}$ is done in Section \ref{s:cellsystem}. The following theorem is the main result of the paper.

\begin{Thm} \label{thm:main}
The process $(\Delta e_i^{a,+}, \, i \ge 1)_{a\ge 0}$ is distributed as $\overline{\mathbf{X}}$.
\end{Thm}

\bigskip

\noindent {\bf Remarks}. 
\begin{itemize}
\item The fact that there is no local explosion (in the sense that there is no compact of $\mathbb{R}\backslash\{0\}$ with infinitely many elements of $\overline{\mathbf{X}}$) can be seen as a consequence of the theorem. 

\item From the skew-product representation of planar Brownian motion, this theorem has an analog in the radial setting. It can be stated as follows. Take a Brownian excursion in the unit disc from boundary to boundary, with continuous determination of its argument (i.e., its winding number around the origin) $z_0>0$. Then, for each $a\ge 0$, record for each excursion made in the disc of radius ${\rm e}^{-a}$ the corresponding winding number. The collection of these winding numbers, ranked in decreasing order of their magnitude and indexed by $a$ is distributed as $\overline{\mathbf{X}}$.

\item One could finally look at the growth-fragmentation associated to the Brownian bubble measure in $\Hb$. It would give an infinite measure on the space of (signed) growth-fragmentation processes starting from $0$. In the non-critical case (i.e. when the natural martingale associated to the intrinsic area converges in $L^1$), a measure on growth-fragmentation processes starting from $0$ has been constructed by Bertoin, Curien and Kortchemski \cite{BCK}, see Section 4.3 there. 
\end{itemize}

\bigskip

\noindent {\bf Related works}. A pure fragmentation process was identified by Bertoin \cite{Ber5} in the case of the linear Brownian excursion where the {\it size} of an excursion was there its duration. Le Gall and Riera \cite{LG-Riera} identified a growth-fragmentation process  in the Brownian motion indexed by the Brownian tree. We will follow the strategy of this paper, making use of excursion theory to prove our theorem.

\bigskip

When killing in $\overline{\mathbf{X}}$ all cells with negative size (and their progeny), one recovers a genuine self-similar (positive) growth-fragmentation driven by $Z$, call it $\mathbf{X}$. The process $\mathbf{X}$ appears in the work of Bertoin et al. \cite{BBCK}, compare Proposition 5.2 in \cite{BBCK} with Proposition \ref{p:cumulant} below.  In Section 3.3 of \cite{BBCK}, the authors exhibit remarkable martingales associated to growth-fragmentation processes and describe the corresponding changes of measure. In the case of $\mathbf{X}$, the martingale consists in summing the sizes raised to the power $5/2$ of all cells alive at time $a$. Under the change of measure, the process $\mathbf{X}$ has a spinal decomposition: the size of the tagged particle is a Cauchy process conditioned on staying positive, while other cells behave normally. In the case of $\overline{\mathbf{X}}$, where we also include cells with negative size, a similar martingale appears, substituting $2$ for $5/2$, while the tagged particle will now follow a Cauchy process (with no conditioning). It is the content of Section \ref{s:martingale}. This martingale is related to the one appearing in \cite{AHS20}, where a change of measure was also specified. In that paper, the authors exhibit a martingale in the radial case, see Section 7.1 there. The martingale in our setting can be viewed as a limit case, where one conformally maps the unit disc to the upper half-plane, then sends the image of the origin towards infinity.

\bigskip

\noindent {\bf Connection with random planar maps}. In \cite{BBCK}, the authors relate a distinguished family of growth-fragmentation processes to the exploration of a  Boltzmann planar map, see Proposition 6.6 there. The mass of a particle in the growth-fragmentation represents the perimeter of a region in the planar map which is currently explored, a negative jump the splitting of the region into two smaller regions to be explored, and a positive jump the discovery of a face with large degree. In this setting, only a negative jump is a birth event. The area of the map is identified as  the limit of a natural martingale associated to the underlying branching random walk, see Corollary 6.7 there. 

On the other hand, a Boltzmann random map can also be seen as the gasket of a $O(n)$ loop model, see Section 8 of \cite{LGM11}. From this point of view, a positive jump of the growth-fragmentation stands for the discovery of a loop which still has to be explored, so that positive jumps will be birth events too. The signed growth-fragmentation $\overline{\mathbf{X}}$ of our paper would represent the   exploration of a planar map decorated with the $O(n)$ model with $n=2$, where the sign depends on the parity of the number of loops which surrounds the explored region. One could wonder whether we would have an intrinsic area  as in \cite{BBCK}. Actually, the natural martingale associated to the branching random walk converges to $0$: it is the so-called critical martingale in the branching random walk literature. The martingale to consider is then the derivative martingale, see Section \ref{s:deriv}, whose limit is proved to be twice the duration of the Brownian excursion (i.e. the inverse of an exponential random variable, see \eqref{gamma tilde}). This gives a conjectured limit of the area of  a $O(2)$ decorated planar map properly renormalized, see \cite{CCM}, Theorem 9, for the analogous results in the $O(n)$ model for $n\neq 2$.

\bigskip

The paper is organized as follows. In Section \ref{s:excursion}, we recall some excursion theory for the planar Brownian motion. Among others, we will define the locally largest fragment, which will be our Eve particle.  In Section \ref{Sec:Markovian}, we show the branching property, identify the law of the Eve particle with that of  $Z$ and exhibit the martingale in our context. Theorem \ref{thm:main} will be proved in Section \ref{s:cellsystem}, where we also show the relation with  \cite{BBCK}. Finally, we identify the limit of the derivative martingale in Section \ref{s:deriv}.

\bigskip

\textbf{Acknowledgements:} We are grateful to Jean Bertoin and Bastien Mallein for stimulating discussions, and to Juan Carlos Pardo for a number of helpful discussions regarding self-similar processes. After a first version of this article appeared online, Nicolas Curien pointed to us the connection with random planar maps, and the link between the duration of the excursion and the area of the map. We warmly thank him for his explanations. We also learnt that Timothy Budd in an unpublished note had already predicted the link between growth-fragmentations of \cite{BBCK} and  planar excursions.



\section{Excursions of Brownian motion in $\Hb$}
\label{s:excursion}

\subsection{The excursion process of Brownian motion in $\Hb$} \label{excursion process}
In this section, we recall some basic facts from excursion theory. Let $(X,Y)$ be a planar Brownian motion defined on the complete probability space $(\Omega, \F, \Pb)$, and $(\F_t)_{t\geq 0}$ be the usual augmented filtration. 

In addition, we call $\mathscr{X}$ the space of real-valued continuous functions $w$ defined on an interval $[0,R(w)]\subset[0,\infty)$, endowed with the usual $\sigma$-fields generated by the coordinate mappings $w\mapsto w(t\wedge R(w))$. Let also $\mathscr{X}_0$ be the subset of functions in $\mathscr{X}$ vanishing at their endpoint $R(w)$. We set $U := \left\{ u=(x,y) \in \mathscr{X}\times \mathscr{X}_0, \; u(0)=0 \; \text{and} \; R(x)=R(y) \right\}$ and $U_{\delta} := U \cup \{\delta\}$, where $\delta$ is a cemetery function and write $U^{\pm}$ for the set of such functions in $U$ with nonnegative and nonpositive imaginary part respectively. These sets are endowed with the product $\sigma-$field denoted $\mathscr{U}_{\delta}$ and the filtration $(\mathcal{F}_t)_{t\ge 0}$ adapted to the coordinate process on $U$. For $u\in U$, we take the obvious notation $R(u):=R(x)=R(y)$. Finally, let $(L_s)_{s\geq 0} = (L^Y_s)_{s\geq 0}$ denote the local time at $0$ of $Y$ and $\tau_s=\tau^Y_s$ its inverse defined by $\tau_s := \inf\{r>0, \; L_r >s\}$. Recall that the set of zeros of $Y$ is almost surely equal to the set of $\tau_s, \tau_{s^-}$ ; we refer to \cite{RY} for more details on local times.

\begin{Def} \label{excursion}
The \emph{excursion process} is the process $e=(e_s, s>0)$ with values in $(U_{\delta},\mathscr{U}_{\delta})$ defined on $(\Omega, \F, \Pb)$ by 
\begin{itemize}
    \item[(i)] if $\tau_s-\tau_{s^-}>0$, then 
    \[e_s : r\mapsto \left(X_{r+\tau_{s^-}}-X_{\tau_{s^-}}, Y_{r+\tau_{s^-}}\right), \quad r\leq \tau_s-\tau_{s^-},\]
    \item[(ii)] if $\tau_s-\tau_{s^-}=0$, then $e_s = \delta$.
\end{itemize}
\end{Def}
Figure \ref{excursion fig} is a (naive) drawing of such an excursion.

\begin{figure} 
\begin{center}
\includegraphics[scale=0.8]{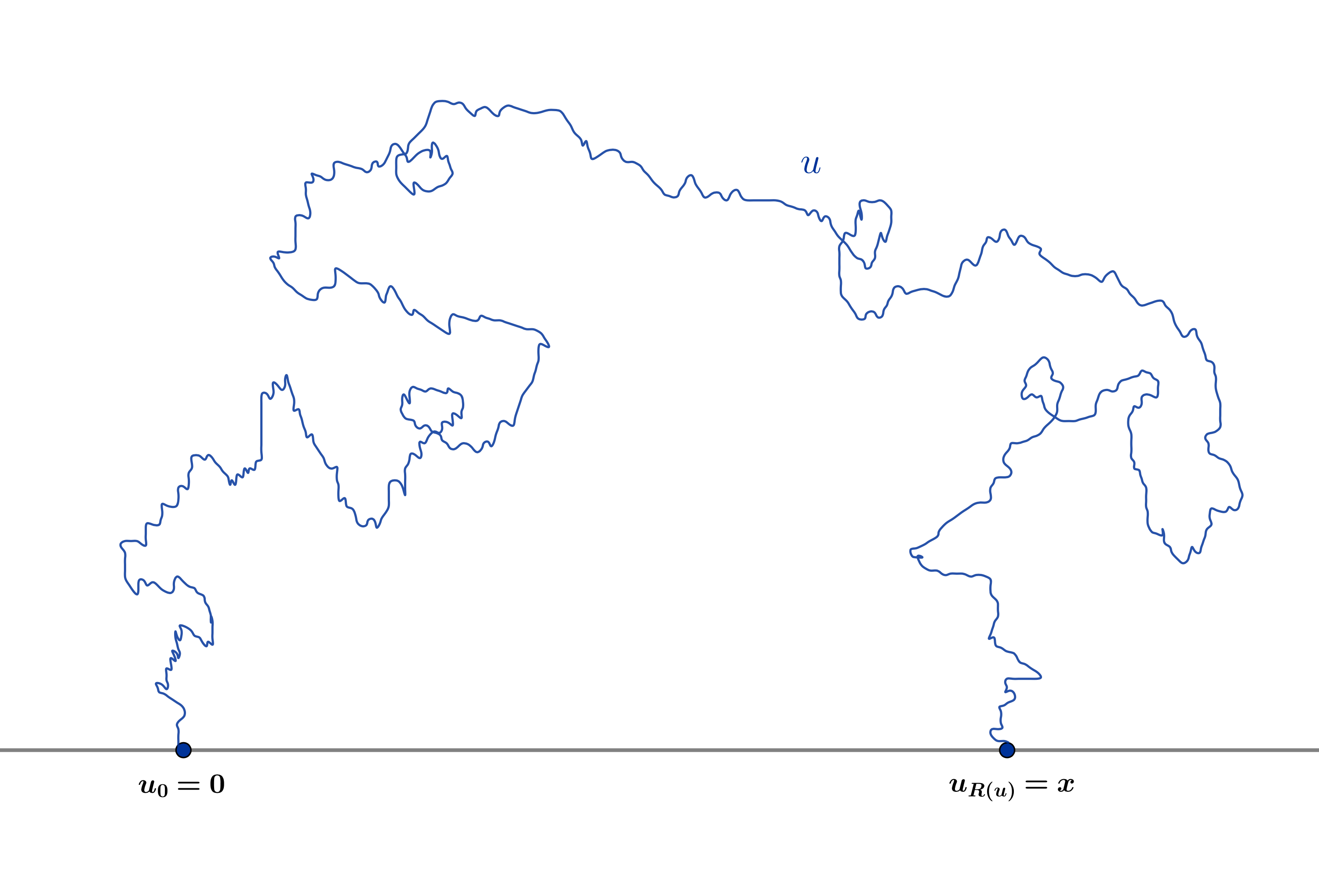}
\end{center}
\caption{Drawing of an excursion in the upper half-plane $\Hb$.}
\label{excursion fig}
\end{figure}

The next proposition follows from the one-dimensional case.
\begin{Prop}
The excursion process $(e_s)_{s>0}$ is a $(\F_{\tau_s})_{s>0}-$Poisson point process.
\end{Prop}
We write $\n$ for the intensity measure of this Poisson point process. It is a measure on $U$, and we shall denote by $\n_+$ and $\n_-$ its restrictions to $U^+$ and $U^-$. We have the following expression for $\n$.
\begin{Prop}
$\n(\mathrm{d}x,\mathrm{d}y) = n(\mathrm{d}y)\Pb(X^{R(y)} \in \mathrm{d}x)$, where $n$ denotes the one-dimensional It\^o's measure on $\mathscr{X}_0$ and $X^T:=(X_t, \, t\in [0,T])$.
\end{Prop}

\subsection{The Markov property under $\n$}
 For any $u\in U$ and any $a>0$, let $T_a := \inf\{0\le t\le R(u), \; y(t)=a\}$ be the hitting time of $a$ by $y$. Then we have the following kind of Markov property under $\n_+$.

\begin{Lem}{(Markov property under $\n$)} \label{Markov under n}

Under $\n_+$, on the event $\{T_a<\infty\}$, the process $\left(u(T_a+t)-u(T_a)\right)_{0\le t\leq R(u)-T_a}$ is independent of $\mathcal{F}_{T_a}$ and has the law of a Brownian motion killed at the time $\rho$ when it reaches $\{\Im(z) = -a\}$. 
\end{Lem}
\begin{proof}
This results from the fact that under the one-dimensional It\^o's measure $n_+$, the coordinate process $t\mapsto y(t)$ has the transition of a Brownian motion killed when it reaches $0$ (cf. Theorem 4.1, Chap. XII in \cite{RY}).

Let $f,g,h_1,h_2$ be nonnegative measurable functions defined on $\mathscr{X}$. For simplicity, write for $w\in\mathscr{X}$ or $w\in C([0,\infty))$, $w(\theta_r) = w(r+\cdot)-w(r)$ and for $T>0$, $w^T:=(w(t),t \in [0,T])$.  We want to compute 
\begin{align*}
    &\int_{U} f(x(\theta_{T_a})) g(y(\theta_{T_a})) h_1(x^{T_a}) h_2(y^{T_a}) \mathds{1}_{\{T_a<\infty\}} \n_+(\mathrm{d}x, \mathrm{d}y) \\
    &= \int_{U} f(x(\theta_{T_a})) g(y(\theta_{T_a})) h_1(x^{T_a}) h_2(y^{T_a}) \mathds{1}_{\{T_a<\infty\}} n_+(\mathrm{d}y) \Pb(X^{R(y)}\in\mathrm{d}x) \\
    &= \int_{\mathscr{X}_0} g(y(\theta_{T_a})) h_2(y^{T_a}) \mathds{1}_{\{T_a<\infty\}} \Eb\left[f\left(\widetilde{X}^{R(y)-T_a(y)}\right) h_1\left(X^{T_a(y)}\right)\right] n_+(\mathrm{d}y) 
\end{align*}
where $\widetilde{X} = X(\theta_{T_a(y)})$, and for $y\in \mathscr{X}_0 $, $T_a=T_a(y)$ is  the hitting time of $a$ by $y$. Using the simple Markov property at time $T_a(y)$ in the above expectation gives 
\begin{align*}
    &\int_{U} f(x(\theta_{T_a})) g(y(\theta_{T_a})) h_1(x^{T_a}) h_2(y^{T_a}) \mathds{1}_{\{T_a<\infty\}} \n_+(\mathrm{d}x, \mathrm{d}y) \\
    &= \int_{\mathscr{X}_0} g(y(\theta_{T_a})) h_2(y^{T_a}) \mathds{1}_{\{T_a<\infty\}}  \Eb\left[f\left(X^{R(y)-T_a(y)}\right)\right]  \Eb\left[h_1\left(X^{T_a(y)}\right)\right] n_+(\mathrm{d}y).  
\end{align*}
Then we can use the Markov property under $n_+$ stated in Theorem 4.1, Chap. XII in \cite{RY}: 
\begin{align*}
    &\int_{U} f(x(\theta_{T_a})) g(y(\theta_{T_a})) h_1(x^{T_a}) h_2(y^{T_a}) \mathds{1}_{\{T_a<\infty\}} \n_+(\mathrm{d}x, \mathrm{d}y) \\
    &= \int_{\mathscr{X}_0}  \Eb\left[h_1\left(X^{T_a(y)}\right)\right] h_2(y^{T_a}) \mathds{1}_{\{T_a<\infty\}} n_+(\mathrm{d}y) \Eb \left[g\left(Y^{T_{-a}}\right)   f\left(X^{T_{-a}}\right) \right]  \\
    &= \int_{U} h_1(x^{T_a}) h_2(y^{T_a}) \mathds{1}_{\{T_a<\infty\}} \n_+(\mathrm{d}x,\mathrm{d}y) \Eb \left[g\left(Y^{T_{-a}}\right)   f\left(X^{T_{-a}}\right) \right].
\end{align*}
This concludes the proof of Lemma \ref{Markov under n}.
\end{proof}


\subsection{Excursions above horizontal levels} \label{excursion level}
We next set some notation for studying the excursions above a given level. Let $a\geq0$ and $u=(x,y)\in U^+$. In the following list of definitions, one should think of $u$ as a Brownian excursion in the sense of Definition $\ref{excursion}$.

Define 
\[\mathcal{I}(a) = \left\{s\in [0,R(u)], \; y(s)>a\right\}.\] 
Then by continuity $\mathcal{I}(a)$ is a countable (possibly empty) union of disjoint open intervals $I_1,I_2,\ldots$ For any such interval $I = \left(i_-,i_+\right)$, take $u_{I}(s) = u(i_- +s)-u(i_-), 0\leq s\leq i_+ -i_-,$ for the restriction of $u$ to $I$, and $\Delta u_I = x(i_+)-x(i_-)$ for the \emph{size} or \emph{length} of $u_I$. Note that $u_I\in U$.

If now $z=u(t)$, $0\leq t\leq R(u)$, is on the path of $u$ and $0\leq a< \Im(z)$, we define 
\[e_a^{(t)} = e_a^{(t)}(u) = u_I,\]
where $I$ is the unique open interval in the above partition of $\mathcal{I}(a)$ such that $t\in I$ (note that this depends on $t$ and not only on $z$, which could be a double point). By convention, we also set for $a=\Im(z)$, $e_a^{(t)} = z$ and $\Delta e_a^{(t)} = 0$. This is represented in an excessively naive way in Figure \ref{excursion-level} below. 

\begin{figure} 
\begin{center}
\includegraphics[scale=0.75]{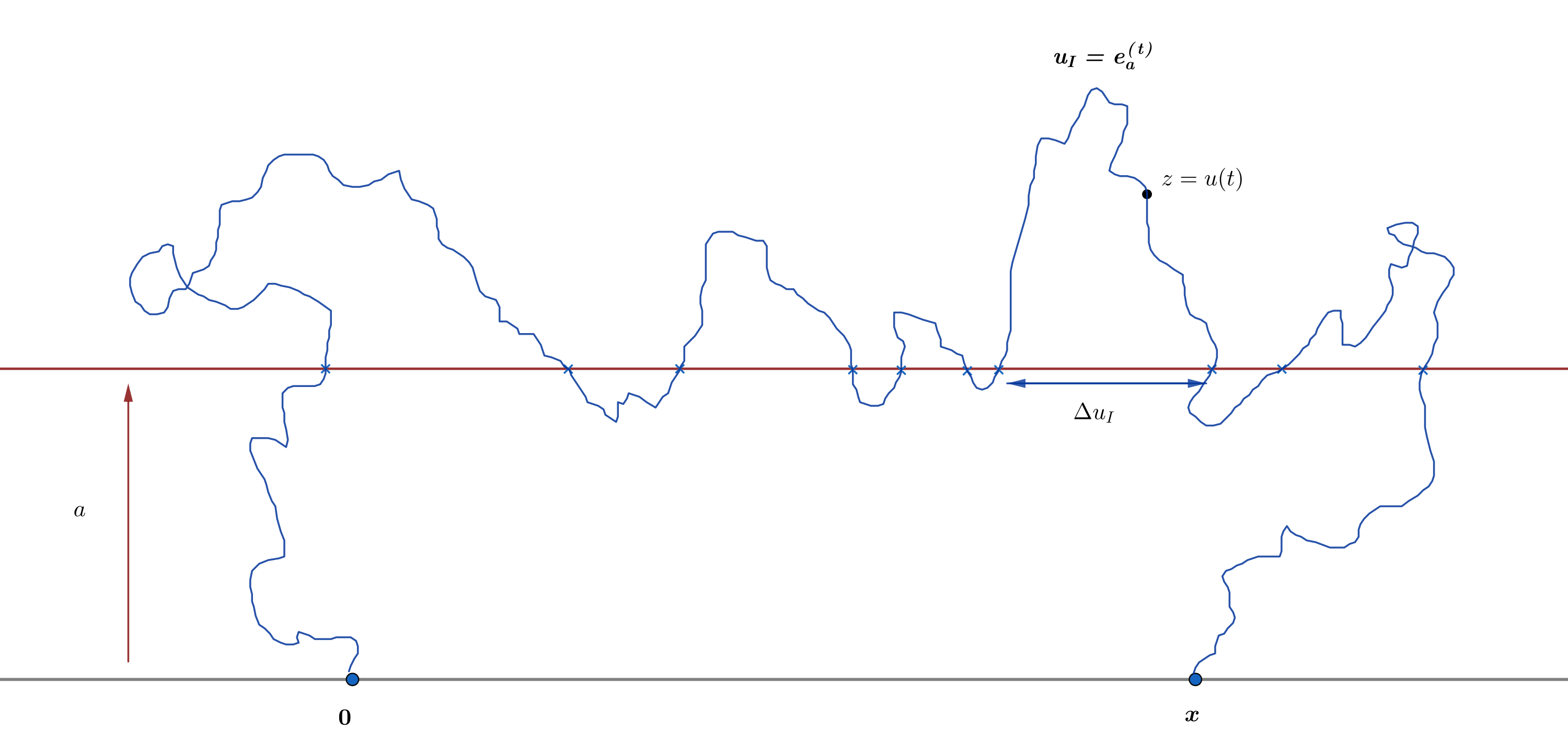}
\end{center}
\caption{Excursions above the level $t$.}
\label{excursion-level}
\end{figure}

For $z=u(t)$, let $F^{(t)}:a\in[0, \Im(z)]\mapsto \Delta e_a^{(t)}$. Define
\begin{align}\label{def:uleft}
u^{t, \leftarrow}
&:=
\left(u(t-s)-u(t)\right)_{0\leq s\leq t},
\\
u^{t,\rightarrow}
&:=
\left(u(t+s)-u(t)\right)_{0\leq s\leq R(u)-t}. \label{def:uright}
\end{align}

\noindent If we set for $a\in [0,y(t)]$,
\begin{align}\label{def:Tleft}
T_a^{t,\leftarrow} 
&:= 
\inf\left\{s\geq 0, \; y(t-s)=a\right\}, 
\\
T_a^{t,\rightarrow} 
&:= 
\inf\left\{s\geq 0, \; y(t+s)=a\right\},
\label{def:Tright}
\end{align}

\noindent we can write $F^{(t)}(a)=u^{t,\rightarrow}(T_a^{t,\rightarrow}) - u^{t,\leftarrow}(T_a^{t,\leftarrow} )$.
\begin{Lem}\label{l:cadlag}
For any $u\in U^+$, for all $0\leq t\leq R(u)$, the function $F^{(t)}$ is c\`adl\`ag.
\end{Lem}

\begin{proof}
Fix $t\in [0,R(u)]$. We want to show that $F^{(t)}$ is c\`adl\`ag on $[0,y(t)]$.  By usual properties of inverse of continuous functions (see Lemma 4.8 and the remark following it in Chapter 0 of Revuz-Yor \cite{RY}), $a\mapsto T_a^{t,\leftarrow}$ and  $a\mapsto T_a^{t,\rightarrow}$ are c\`adl\`ag (in $a$). Hence $F^{(t)}$ is c\`adl\`ag since $u$ is continuous.  
\end{proof}


\subsection{Bismut's description of It\^o's measure in $\Hb$}
In the case of one-dimensional It\^o's measure $n_+$, Bismut's description roughly states that if we pick an excursion $u$ at random according to $n_+$, and some time $0\leq t\leq R(u)$ according to the Lebesgue measure, then the "law" of $u(t)$ is the Lebesgue measure and conditionally on $u(t)=\alpha$, the left and right parts of $u$ (seen from $u(t)$) are independent Brownian motions killed at $-\alpha$ (see Theorem 4.7, Chap. XII in  \cite{RY}). We deduce an analogous result in the case of It\^o's measure in $\Hb$ and we apply it to show that for $\n_+-$almost every excursion, there is no loop remaining above any horizontal level. 
\begin{Prop} \label{Bismut} (Bismut's description of It\^o's measure in $\Hb$)

Let $\overline{\n}_+$ be the measure defined on $\R_+\times U^+$ by
\[\overline{\n}_+(\mathrm{d}t,\mathrm{d}u) = \mathds{1}_{\{0\leq t\leq R(u)\}} \mathrm{d}t \, \n_+(\mathrm{d}u).\]
Then under $\overline{\n}_+$ the "law" of $(t,(x,y))\mapsto y(t)$ is the Lebesgue measure $\mathrm{d}\alpha$ and conditionally on $y(t)=\alpha$, $u^{t, \leftarrow}=\left(u(t-s)-u(t)\right)_{0\leq s\leq t}$ and $u^{t,\rightarrow}=\left(u(t+s)-u(t)\right)_{0\leq s\leq R(u)-t}$ are independent Brownian motions killed when reaching $\left\{\Im(z)=-\alpha\right\}$.
\end{Prop}

See Figure \ref{Bismut fig}. Proposition \ref{Bismut} is a direct consequence of the one-dimensional analogous result, for which we refer to \cite{RY} (see Theorem 4.7, Chapter XII).

The next proposition ensures that for almost every excursion under $\n_+$, there is no loop growing above any horizontal level. Let 
\[\mathscr{L} := \{u\in U^+, \; \exists 0\leq t\leq R(u), \; \exists 0\leq a<y(t), \; \Delta e_a^{(t)}(u) = 0\},\]
be the set of excursions $u$ having a loop remaining above some level $a$.
Then we have : 
\begin{Prop} \label{loop}
\[\n_+\left(\mathscr{L}\right) = 0.\]
\end{Prop}
\begin{proof}
We first prove the result under $\overline{\n}_+$, namely
\[\overline{\n}_+\left(\{(t,u)\in \R_+\times U^+, \;  \; \exists 0\leq a<y(t), \; \Delta e_a^{(t)}(u) = 0\}\right) = 0.\]
Recall the notation \eqref{def:uleft}-\eqref{def:Tright}. From Bismut's description of $\overline{\n}_+$ we get
\begin{align*}
&\overline{\n}_+\left(\{(t,u)\in \R_+\times U^+, \;  \; \exists 0\leq a<y(t), \; \Delta e_a^{(t)}(u) = 0\}\right) \\
&= \overline{\n}_+\left(\{(t,u)\in \R_+\times U^+, \;  \; \exists 0\leq a<y(t), \; u^{t,\rightarrow}(T_a^{t,\rightarrow}) = u^{t,\leftarrow}(T_a^{t,\leftarrow})\}\right) \\
&= \int_0^{\infty} \mathrm{d}\alpha \, \Pb\left(\exists 0< a\le \alpha, \, X_{T_{a}} = X'_{T'_{a}}\right),
\end{align*}
where $X$ and $X'$ are independent linear Brownian motions, and $T_{a}$ and $T'_{a}$ are hitting times of $a$ of other independent Brownian motions (corresponding to the imaginary parts). Now, $X_{T_{a}}$ and $X'_{T'_{a}}$ are independent symmetric Cauchy processes, and therefore $X_{T_{a}} - X'_{T'_{a}}$ is again a Cauchy process (see Section 4, Chap. III of \cite{RY}). Since points are polar for the symmetric Cauchy process (see \cite{Ber4}, Chap. II, Section 5), we obtain $\Pb\left(\exists 0< a\le \alpha, \, X_{T_{a}} = X'_{T'_{a}}\right)=0$ and under $\overline{\n}_+$ the result is proved. 

\begin{figure} 
\begin{center}
\includegraphics[scale=0.7]{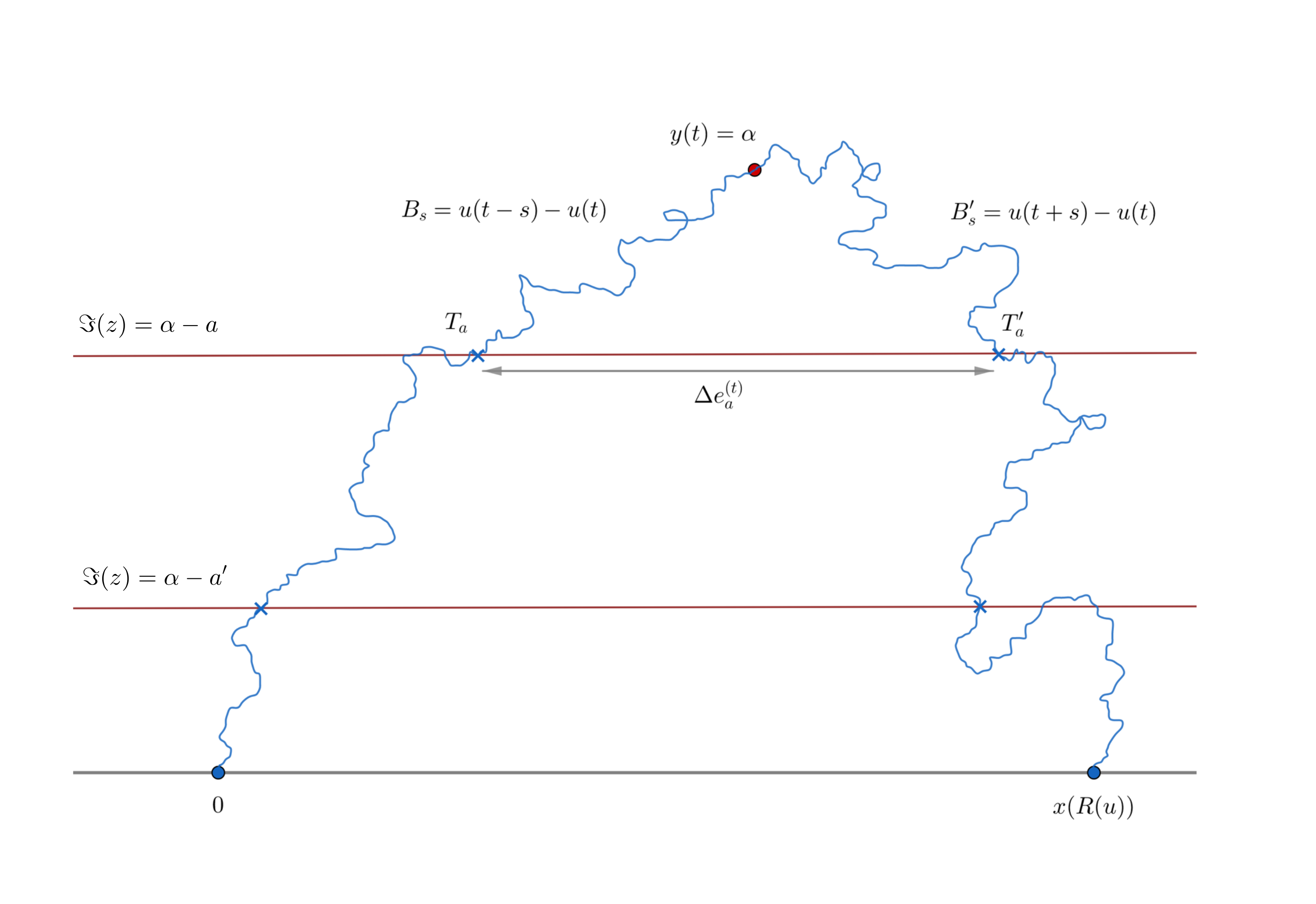}
\end{center}
\caption{Bismut's description of $\n_+$}
\label{Bismut fig}
\end{figure}

To extend the result to $\n_+$, we notice that if $u\in \mathscr{L}$, then the set of $t$'s satisfying the definition of $\mathscr{L}$ has positive Lebesgue measure: namely, it contains all the times until the loop comes back to itself. This translates into 
\[\mathscr{L} \subset \left\{u\in U^+, \, \int_{0}^{R(u)} \mathds{1}_{\{\exists 0\leq a<y(t), \; \Delta e_a^{(t)}(u) = 0\}}\mathrm{d}t >0\right\}.\]
But 
\begin{align*}
&\n_+\left(\int_{0}^{R(u)} \mathds{1}_{\{\exists 0\leq a<y(t), \; \Delta e_a^{(t)}(u) = 0\}}\mathrm{d}t\right) \\
&= \int_{U^+} \int_{0}^{R(u)} \mathds{1}_{\{\exists 0\leq a<y(t), \; \Delta e_a^{(t)}(u) = 0\}}\mathrm{d}t \, \n_+(\mathrm{d}u) \\
&= \overline{\n}_+\left(\{(t,u)\in \R_+\times U^+, \;  \; \exists 0\leq a<y(t), \; \Delta e_a^{(t)}(u) = 0\}\right).
\end{align*}
Hence, by the first step of the proof, 
\[\n_+\left(\int_{0}^{R(u)} \mathds{1}_{\{\exists 0\leq a<y(t), \; \Delta e_a^{(t)}(u) = 0\}}\mathrm{d}t\right) = 0,\]
which gives $\displaystyle \int_{0}^{R(u)} \mathds{1}_{\{\exists 0\leq a<y(t), \; \Delta e_a^{(t)}(u) = 0\}}\mathrm{d}t=0$ for $\n_+-$almost every excursion, and the desired result.
\end{proof}


\subsection{The locally largest excursion}
In \cite{BBCK}, the authors give a canonical way to construct the growth-fragmentation, through the so-called locally largest fragment. We want to mimic this construction in our case.

\bigskip

\noindent In order to define the locally largest excursion, we set for $u \in U^+$ and $0\le t \le R(u)$,
\[
S(t) := \sup\left\{ a\in [0,y(t)], \quad \forall \, 0\leq a'\leq a, \; \big|F^{(t)}(a')\big| \geq \big|F^{(t)}(a'^-)-F^{(t)}(a')\big|\right\}.
\]
Observe that the supremum is taken over a non-empty set by Lemma \ref{l:cadlag} as soon as $y(t)>0$ and $u(R(u))\neq 0$. Let 
\[ S := \underset{0\leq t\leq R(u)}{\sup} S(t).\]
In the case of Brownian excursions, the following proposition holds.

\begin{Prop} \label{locally largest prop}
For almost every $u$ under $\n_+$, there exists a unique $0\leq \tb\leq R(u)$ such that $S(\tb) = S$. Moreover, $S=\Im(\zb)$ where $\zb=u(\tb)$.

We call $\left(e_a^{(\tb)}\right)_{0\leq a\leq \Im(\zb)}$ the \emph{locally largest excursion} and $\left(\Xi(a) =\Delta e_a^{(\tb)}\right)_{0\leq a\leq \Im(\zb)}$ the \emph{locally largest fragment}.
\end{Prop}

Thus $\Xi$ is the length of the excursion which is locally the largest, meaning that at any level $a$ where the locally largest excursion splits, $\Xi$ is larger (in absolute value) than the length of the other excursion. See Figure \ref{locally largest} for a picture of $\zb$. Following \cite{BBCK}, we will see it as the Eve particle of our growth-fragmentation process.

\begin{figure} 
\begin{center}
\includegraphics[scale=0.75]{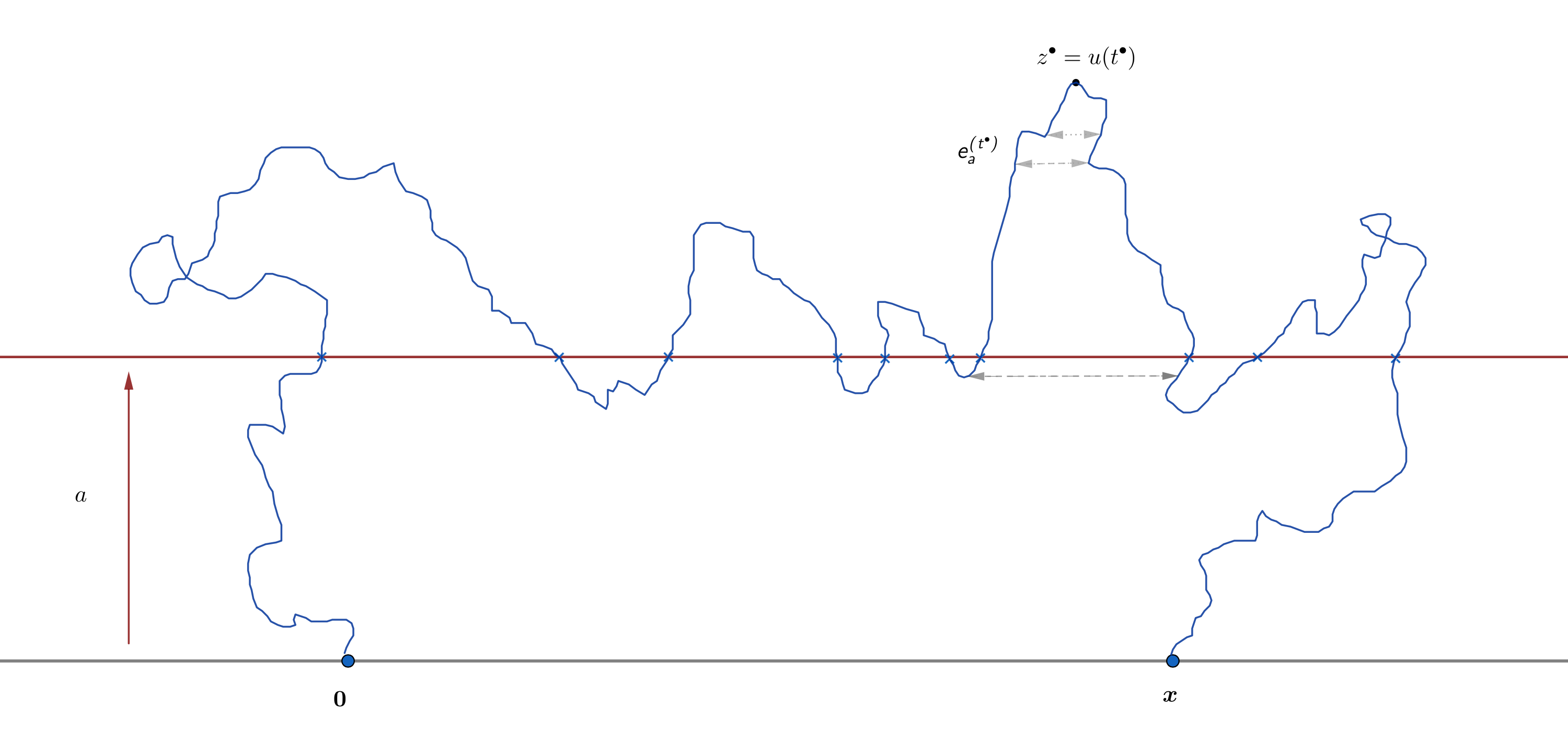}
\end{center}
\caption{The locally largest excursion.}
\label{locally largest}
\end{figure}

\begin{proof} 
\emph{Existence.} We deal with the excursions $u$ satisfying the following properties, which happen $\n_+$-almost everywhere : $u$ has no loop above any horizontal level (see Proposition \ref{loop}) and $y$ has distinct local minima. Take a convergent sequence $(t_n, n\ge 1)$ such that $S(t_n)$ converges to $S$, and denote by $\tb$ the limit of $t_n$. We have necessarily, by definition of $S(t)$, that $y(t_n)\geq S(t_n)$. By continuity of $y$, we get that $y(\tb) \ge S$. Take $a<S$. For $n$ large enough, since $a<y(\tb)$, we observe that $t_n$ and $\tb$ are in the same excursion above $a$, i.e. $e_a^{(\tb)}=e_a^{(t_n)}$. For such $n$, $F^{(t_n)}(a')=F^{(\tb)}(a')$ for all $a'\leq a$. Moreover, for $n$ large enough, $S(t_n)>a$ hence for all $a'\le a$, $\big|F^{(\tb)}(a')\big|=\big| F^{(t_n)}(a') \big| \geq \big|F^{(t_n)}(a'^-)-F^{(t_n)}(a')\big|=\big|F^{(\tb)}(a'^-)-F^{(\tb)}(a')\big|$. It implies that $S(\tb)\geq a$, hence $S(\tb)\geq S$ by taking $a$ arbitrarily close to $S$. We found $\tb$ such that $S(\tb)=S$.

We show that $y(\tb)=S$.  Notice that, for all $0\leq t\leq R(u)$, by right-continuity of $F^{(t)}$, the set
\[
A(t):=
\left\{
    0\leq a\leq y(t), \quad \forall \, 0\leq a'\leq a, \; \big|F^{(t)}(a')\big| \geq \big|F^{(t)}(a'^-)-F^{(t)}(a')\big|
\right\}
\]
is open in $[0,y(t)]$. Indeed, for $a<y(t)$, $e_a^{(t)}$ cannot be an excursion with size $0$ by assumption, and so by right-continuity, we can take  $\delta>0$ such that on $[a,a+\delta]$, $F^{(t)}$ takes values in $\left(\frac34 F^{(t)}(a), \frac32 F^{(t)}(a)\right)$ (in the case $F^{(t)}(a)>0$, without loss of generality). For such a $\delta$, and for any $a'\in[a,a+\delta]$, $F^{(t)}(a')>\frac34 F^{(t)}(a) > \frac34 \frac23 F^{(t)}(a'^-) =\frac12 F^{(t)}(a'^-)$, and $F^{(t)}(a'^-)\ge 0$. These two inequalities imply that $|F^{(t)}(a')| \geq |F^{(t)}(a'^-)-F^{(t)}(a')|$.

Now suppose that $S<y(\tb)$ and let us find a contradiction. We have $A(\tb)=[0,S)$, hence $ |F^{(\tb)}(S)| < |F^{(\tb)}(S^-)-F^{(\tb)}(S)| $. Write  $e_a^{(\tb)}=u_I$ with $I=(i_{a,-},i_{a,+})$, so that $F^{(\tb)}(a)=x(i_{a,+})-x(i_{a,-})$. Since $F^{(\tb)}$ jumps at $S$, either $i_{\cdot,-}$ or $i_{\cdot,+}$ jumps at $S$. Both cases cannot happen at the same time because local minima of $y$ are all distinct. Suppose  for example that $i_{S^-,-}<i_{S,-}$. Take $t\in (i_{S^-,-},i_{S,-})$ (see Figure \ref{locally largest proof}). We have $F^{(t)}(a)=F^{(\tb)}(a)$ for all $a<S$ and 
\begin{align*}
    F^{(t)}(S)
&=
    x(i_{S,-})-x(i_{S^-,-})= x(i_{S^-,+})-x(i_{S^-,-})-( x(i_{S,+})-x(i_{S,-})) \\
&=
     F^{(\tb)}(S^-) - F^{(\tb)}(S) \\
     &= F^{(t)}(S^-) - F^{(\tb)}(S). 
\end{align*}

\begin{figure} 
\begin{center}
\includegraphics[scale=0.75]{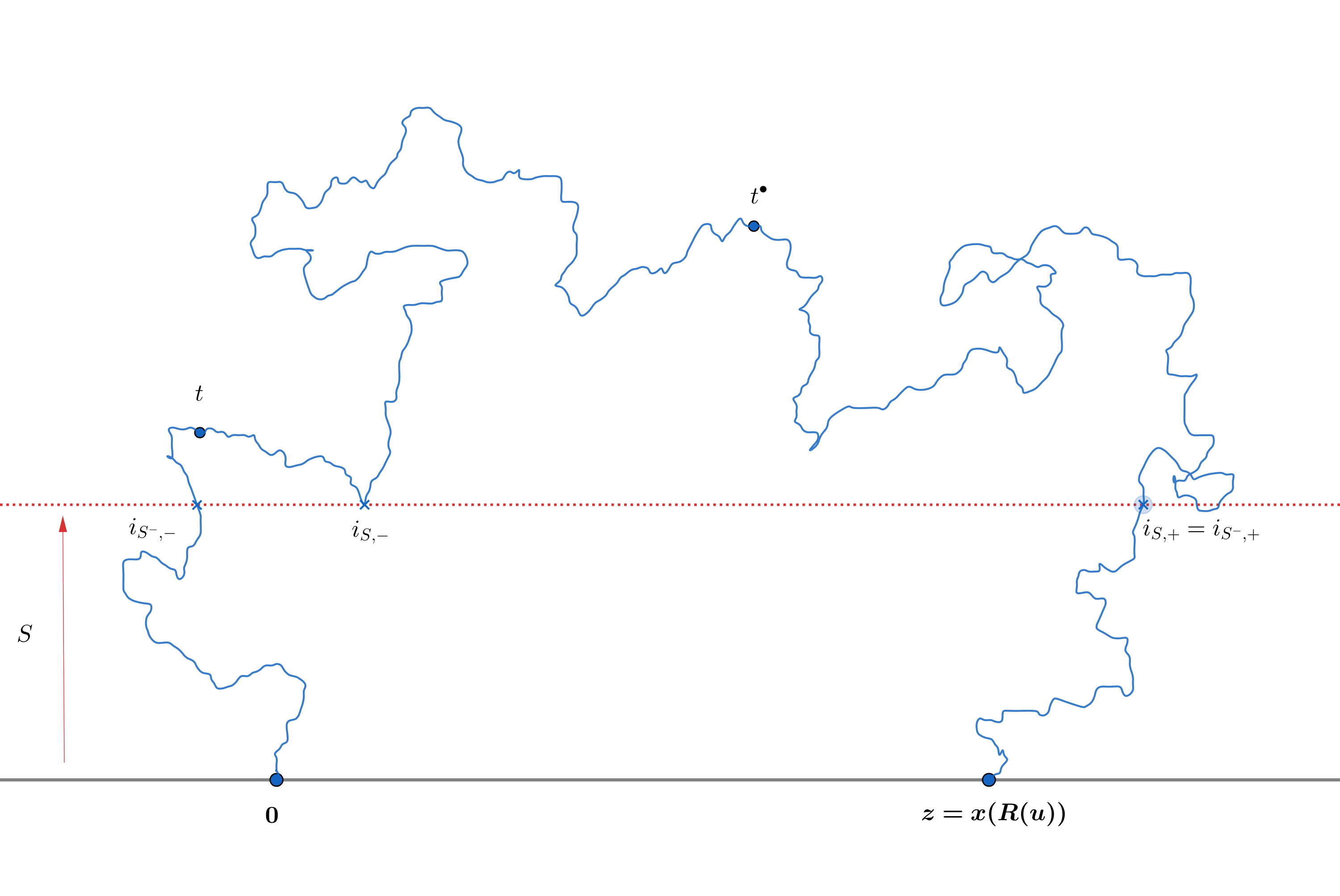}
\end{center}
\caption{Construction of the locally largest excursion.}
\label{locally largest proof}
\end{figure}

We deduce that $|F^{(t)}(S)|=|F^{(\tb)}(S^-)-F^{(\tb)}(S)| >|F^{(\tb)}(S)|=|F^{(t)}(S^-) - F^{(t)}(S) |$. Then $A(t)$ is open in $[0,y(t)]$, contains $S$, and we have $y(t)>S$.  Hence $\sup A(t)>S $ which gives the desired contradiction.

\emph{Uniqueness.} Suppose that $S(t)=S(t')=S$ with $t<t'$ and let us find again a contradiction. We showed that necessarily, $y(t)=y(t')=S$. Let $t_m \in[t,t']$ such that $y(t_m) = \min\left\{ y(r),\, r\in [t,t']   \right\}$. Set $a_m:=y(t_m)$. Observe that $t$ and $t'$ cannot be starting times or ending times of an excursion of $y$ (otherwise we could have extended the locally largest fragment inside this excursion for some positive height). Hence $a_m< S$. At level $a_m$, there must be a splitting into two excursions (one straddling time $t$, the other $t'$) with equal size. It happens on a negligible set under $\n_+$. To see it, we can restrict to $t<t'$ rationals and use the Markov property at time $t'$. 
\end{proof}


\subsection{Disintegration of It\^o's measure over the size of the excursions}
We are interested in conditioning It\^o's measure of excursions in $\Hb$ on their initial size, \emph{i.e.} in fixing the value of $x(R(u))=z$. This will allow us to define probability measures $\gamma_z$ which disintegrate $\n_+$ over the value of the endpoint $z$. Properties will simply transfer from $\n_+$ to $\gamma_z$ \emph{via} the disintegration formula. Define $\Pbb{r}{a}{b}$ as the law of the one-dimensional Brownian bridge of length $r$ between $a$ and $b$, and $\Pi_r$ as the law of a three-dimensional Bessel ($BES^3$) bridge of length $r$ from $0$ to $0$.  

\begin{Prop} \label{disintegration prop}
We have the following disintegration formula
\begin{equation} \label{disintegration}
    \n_+ = \int_{\R} \frac{\mathrm{d}z}{2\pi z^2} \, \gamma_z,
\end{equation}
where for $z\ne 0$,
\begin{equation} \label{gamma tilde}
    \gamma_z = \int_{\R_+} \mathrm{d}v \, \frac{e^{-1/2v}}{2v^2} \,  \Pbb{vz^2}{0}{z}\otimes \Pi_{vz^2}. 
\end{equation}
\end{Prop}
\begin{proof}
Let $f$ and $g$ be two nonnegative measurable functions defined on $\mathscr{X}$ and $\mathscr{X}_0$ respectively. Thanks to It\^o's description of $n_+$ (see \cite{RY}, Chap. XII, Theorem 4.2), we have
\begin{align*}
    \int_{U} f(x)g(y) \, \n_+(\mathrm{d}x,\mathrm{d}y) &= \int_{U} f(x)g(y) \, n_+(\mathrm{d}y) \Pb\left(X^{R(y)}\in \mathrm{d}x\right) \\
    &=  \int_{\R_+} \frac{\mathrm{d}r}{2\sqrt{2\pi r^3}} \int_{\mathscr{X}} f(x) \, \Pi_r[g] \,\Pb\left(X^{r}\in \mathrm{d}x\right).
\end{align*}
Now, decomposing on the value of the Gaussian r.v. $X_r$ yields 
\[\int_{U} f(x)g(y) \, \n_+(\mathrm{d}x,\mathrm{d}y) = \int_{\R_+} \frac{\mathrm{d}r}{2\sqrt{2\pi r^3}} \int_{\R} \mathrm{d}z \frac{e^{-z^2/2r}}{\sqrt{2\pi r}}  \Pi_r[g] \, \Ebb{r}{0}{z} \left[f\right].\]
We finally perform the change of variables $v(r)=r/z^2$ to get
\[\int_{U} f(x)g(y) \, \n_+(\mathrm{d}x,\mathrm{d}y) = \int_{\R} \frac{\mathrm{d}z}{2\pi z^2} \int_{\R_+} \mathrm{d}v \, \frac{e^{-1/2v}}{2v^2}  \Ebb{vz^2}{0}{z} \left[f\right] \, \Pi_{vz^2}[g].\]
\end{proof}

\begin{Lem}\label{l:scaling}
Let $z$ be a nonzero real number. The image measure of $\gamma_z$ by the function which sends $(x,y)$ to
\[
 \left( \frac{x(tz^2)}{z}, \frac{y(t z^2)}{|z|} \right), \quad 0\le t\le \frac{R(u)}{z^2},
\]
is $\gamma_1$. 
\end{Lem}
\begin{proof}
It comes from the definition of $\gamma_z$ and the scaling property of $BES^3$ bridge and Brownian bridge.
\end{proof}

\subsection{The metric space of excursions in $\Hb$} \label{Sec:metric space}
Very often, results under $\gamma_z$ can be obtained by proving the analog under the It\^o's measure $\n_+$, and then disintegrating over $z=x(R(u))$. This usually provides results under $\gamma_z$ for Lebesgue-almost every $z>0$, and so we would like to study the continuity of $z\mapsto \gamma_z$. This requires to define a topology on the space of excursions $U^+$. All these results will be stated for $z>0$ because the scaling depends on the sign of the endpoint (Lemma \ref{l:scaling}), but they all extend to the general case.

We therefore introduce the usual distance 
\[d(u,v) = |R(u)-R(v)| + \sup_{t\ge 0} |u(t\wedge R(u))-v(t\wedge R(v))|,\]
where we identified $\delta$ with the excursion with lifetime $0$. The distance $d$ makes $U^+$ into a Polish space. The following lemmas may come in useful. 

\begin{Lem} \label{lem:continuity delta}
The map $\Delta:u\in U^+ \mapsto \Delta u = x(R(u))$ is continuous. 
\end{Lem}
\begin{proof}
This is straightforward since $|x(R(u))-x'(R(u'))| = |u(R(u))-u'(R(u'))| \le d(u,u')$ for $u=(x,y)$ and $u'=(x',y')$.
\end{proof}
\begin{Lem} \label{lem:continuity}
Let $u\in U^+$. Then $z\in\R^*_+\mapsto u^{(z)}:=z u(\cdot/z^2) = \left(z u(t/z^2), 0\le t\le R(u)z^2\right)$ is a continuous function.
\end{Lem}
\begin{proof}
Let $z_0>0$. Then for all $z>0$
\[ d(u^{(z)}, u^{(z_0)}) = R(u)|z^2-z_0^2| + \sup_{t\ge 0} \left|zu\left(\frac{t}{z^2}\wedge R(u)\right)-z_0u\left(\frac{t}{z_0^2}\wedge R(u)\right)\right|. \]
The second term is  
\begin{align*} 
&\sup_{t\ge 0} \left|zu\left(\frac{t}{z^2}\wedge R(u)\right)-z_0u\left(\frac{t}{z_0^2}\wedge R(u)\right)\right| \\
&\le z \sup_{t\ge 0} \left|u\left(\frac{t}{z^2}\wedge R(u)\right)-u\left(\frac{t}{z_0^2}\wedge R(u)\right)\right| + \sup_{t\ge 0} \left| (z-z_0) u\left(\frac{t}{z^2}\wedge R(u)\right)\right| \\
&\le z \sup_{t\ge 0} \left|u\left(\frac{t}{z^2}\wedge R(u)\right)-u\left(\frac{t}{z_0^2}\wedge R(u)\right)\right| + |z-z_0| \sup_{t\ge 0} |u(t)|.
\end{align*}
We conclude by using the uniform continuity of $u$.
\end{proof}

\bigskip

If we equip the set $\mathcal{P}(U^+)$ of probability measures on $U^+$ with the topology of weak convergence, we have the following result.
\begin{Prop} \label{prop: continuity gamma}
The map $z\in \R_+^*\mapsto \gamma_z$ is continuous.
\end{Prop}
\begin{proof}
Let $G$ be a continuous bounded function on $U^+$. Then by scaling (Lemma \ref{l:scaling}), for all $z>0$,
\[\gamma_z(G) = \gamma_1 \left[ G(u^{(z)})\right]. \]
Applying Lemma \ref{lem:continuity} together with the dominated convergence theorem yields the desired result.
\end{proof}

\bigskip

Also, we will use the continuity of the excursions cut at horizontal levels. 
Recall from Section \ref{excursion level} that ${\mathcal I}(a)$ is the set of times when the excursion $u\in U^+$ lies above $a$, and for each  connected component $I$ of ${\mathcal I}(a)$, $u_I$ denotes the associated excursion above $a$. The path $u_I$ is an excursion above $a$, $I$ is the \emph{time interval} of $u_I$, and the \emph{size} or \emph{length} of $u_I$ is the difference between its endpoint and its starting point.

On $\{T_a<\infty\}$, we rank the excursions above $a$ according to the absolute value of their \emph{size}.  Write $z^{a,+}_1 = z^{a,+}_1(u), z^{a,+}_2 = z^{a,+}_2(u), \ldots$ for the sizes, ranked in descending order of their absolute value, and $e^{a,+}_1 = e^{a,+}_1(u), e^{a,+}_2 = e^{a,+}_2(u), \ldots$ for the corresponding excursions. This is possible since for any fixed $\varepsilon>0$ there are only finitely many excursions with length larger than $\varepsilon$ in absolute value. 

\begin{Prop} \label{prop:continuity excursions}
Let $a>0$ and $z>0$. For any $i\ge 1$, the function $e^{a,+}_i$ is continuous on $U^+$ on the event $\{T_a<\infty\}$ outside a $\gamma_z$-negligible set. 
\end{Prop}
\begin{proof}
We consider the set $\mathscr{E}$ of trajectories $u=(x,y)$ such that $T_a<\infty$ and satisfying the following conditions, which occur with $\gamma_z$-probability one when conditioned on touching $a$: the level $a$ is not a local minimum for $y$, there exist infinitely many excursions above $a$, all excursions touch $a$ only at their starting point and endpoint, the sizes $(z_i^{a,+},\, i\ge 1)$ of the excursions are all distinct. Let $i\ge 1$ and $u=(x,y)\in \mathscr{E}$. We want to show that $e_i^{(a,+)}$ is continuous at $u$. 

Let $t$ be a time in the excursion $e_i^{a,+}$, i.e. such that $y(t)>a$ and $e^{(t)}_a=e_i^{a,+}$. We restrict our attention to $u'=(x',y')\in \mathscr{E}$ close enough to $u$ so that $y'(t)>a$ and we will write $e_a'^{(t)}$ for the excursion of $u'$ corresponding to $t$. Let $\varepsilon>0$.
 
 \begin{itemize}
 \item \emph{First, we want to find $\delta>0$ such that, whenever $d(u,u')<\delta$, the durations of the excursions $e_a^{(t)}$ and $e_a'^{(t)}$ are close, namely $|R(e_a'^{(t)})-R(e^{(t)}_a)|<\varepsilon$.}
Write $(i_-(a), i_+(a))$, and $(i'_-(a), i'_+(a))$, for the excursion time intervals corresponding to $e_a^{(t)}$ and $e_a'^{(t)}$ respectively. For simplicity, we take the notation $R=R(e^{(t)}_a)$ and $R'=R(e'^{(t)}_a)$. Since $a$ is not a local minimum for $y$, there exist times $t_1 \in (i_-(a)-\frac{\varepsilon}{2},i_-(a))$  and $t_2 \in (i_+(a),i_+(a) + \frac{\varepsilon}{2})$ when $y$ is strictly below $a$. Take $\delta_1\in (0,a)$ such that $y(t_1) $ and $y(t_2)$ are in $(0,a-\delta_1)$. Let $u'=(x',y')\in \mathscr{E}$ such that $d(u,u')< \frac{\delta_1}{2}$. We deduce that $y'(t_1)< y(t_1)+ \frac{\delta_1}{2}<a$ and similarly $y'(t_2)<a$. This implies that $i'_-(a)\ge t_1>i_-(a)-\frac{\varepsilon}{2}$ and $i'_+(a)\le t_2<i_+(a)+\frac{\varepsilon}{2}$. Likewise, pick two times $t_3 \in (i_-(a),i_-(a) + {\varepsilon \over 2})$ and $t_4 \in (i_+(a)-\frac{\varepsilon}{2},i_+(a))$ such that $t_3<t<t_4$. Since the excursion $e_a^{(t)}$ touches level $a$ only at its extremities, the distance between the compact $u([t_3,t_4])$ and the closed set $\{\Im(z) = a\}$ is positive, and so, on the interval $[t_3,t_4]$, $y$ remains above, say, $a+\delta_2$ where $\delta_2>0$. 
Then when $d(u,u')<\frac{\delta_2}{2}$, the excursion $e_a'^{(t)}$ will satisfy $i'_-(a)<t_3< i_-(a)+\frac{\varepsilon}{2}$ and $i'_+(a)>t_4> i_+(a)-\frac{\varepsilon}{2}$. Therefore, when $d(u,u')<\delta=\min(\frac{\delta_1}{2},\frac{\delta_2}{2}) $, we get that $|i'_-(a)- i_-(a)|<\frac{\varepsilon}{2}$ and $|i'_+(a)- i_+(a)|<\frac{\varepsilon}{2}$, so in particular $|R'-R|<\varepsilon$. Observe that we not only proved that the durations are close, but also that the times $i_-, i'_-$ (and $i_+, i'_+$) are close, and this will be useful in the remainder of the proof. 

\item \emph{Secondly, we show that we can take $\delta'>0$ small enough so that \[\sup_{s\ge 0} |e_a^{(t)}(s\wedge R)-e_a'^{(t)}(s\wedge R')|<\varepsilon,\] whenever $d(u,u')<\delta'$.}

Take $\eta=\eta(\varepsilon) >0$ some modulus of uniform continuity of $u$ with respect to $\varepsilon$. The previous paragraph gives the existence of $\delta>0$ such that when $u'\in \mathscr{E}$ and $d(u,u')<\delta$, $|i'_-(a)- i_-(a)|<\eta/3$ and $|i'_+(a)- i_+(a)|<\eta/3$. Without loss of generality, we can assume that $\delta<\varepsilon$. Define $\delta'= \min(\delta,\eta)$, and let $u'\in \mathscr{E}$ such that $d(u,u')<\delta'$. For all $s\ge 0$, we have 
\begin{align}
&|e_a^{(t)}(s\wedge R)-e_a'^{(t)}(s\wedge R')|  \notag \\
&= \left|u(i_-(a)+(s\wedge R))-u(i_-(a))- u'(i'_-(a)+(s\wedge R'))+u'(i'_-(a))\right| \notag \\
&\le \left|u(i_-(a))-u'(i'_-(a))\right| + \left|u(i_-(a)+(s\wedge R))- u(i'_-(a)+(s\wedge R'))\right|. \label{eq:continuity}
\end{align}
Now, 
\[
\left|u(i_-(a))-u'(i'_-(a))\right| \le \left|u(i_-(a))-u(i'_-(a))\right| + \left|u(i'_-(a))-u'(i'_-(a))\right|, 
\]
and so by uniform continuity of $u$ and because $d(u,u')<\delta'<\varepsilon$, we obtain 
\begin{equation}
    \left|u(i_-(a))-u'(i'_-(a))\right| \le 2\varepsilon. \label{eq:continuity1}
\end{equation}
Similarly, the second term of (\ref{eq:continuity}) is 
\begin{align*}
  &\left|u(i_-(a)+(s\wedge R))- u'(i'_-(a)+(s\wedge R'))\right| \\
  &\le \left|u(i_-(a)+(s\wedge R))- u(i'_-(a)+(s\wedge R'))\right| \\
  & \quad + \left|u(i'_-(a)+(s\wedge R'))- u'(i'_-(a)+(s\wedge R'))\right|,
\end{align*}
and since $|i_-(a)+(s\wedge R)-i'_-(a)-(s\wedge R')|<\eta$, we can conclude in the same way that
\begin{equation}
    \left|u(i_-(a)+(s\wedge R))- u'(i'_-(a)+(s\wedge R'))\right|\le 2\varepsilon. \label{eq:continuity2}
\end{equation}
Inequalities (\ref{eq:continuity}), (\ref{eq:continuity1}) and (\ref{eq:continuity2}) give 
\[|e_a^{(t)}(s\wedge R)-e_a'^{(t)}(s\wedge R')|\le 4 \varepsilon,\]
which is the desired result.
\end{itemize}
So far, we proved that $e_a^{(t)}$ is continuous at $u$. To conclude, we need an argument to say that this is the $i$-th excursion above $a$ for $u'$ sufficiently close to $u$.

\begin{itemize}
\item \emph{Finally, we show that we can take $\delta''>0$ small enough so that $e_i'^{a,+} = e_a'^{(t)}$ whenever $d(u,u')<\delta''$.}

This is derived in two steps.
\begin{itemize}
    \item[-] \underline{Step 1:} Let $\eta>0$, and introduce, for $u'\in \mathscr{E}$, the number $N_{\eta}(u')$ of time intervals $(i_-,i_+)$ of excursions of $u'$ above $a$ such that $i_+-i_-> \eta$. Note that $N_{\eta}(u')\le \frac{R(u')}{\eta}<\infty$. We take $\eta$ such that $u$ has no excursion time interval above $a$  satisfying $i_+-i_-=\eta$. The first step consists in proving that for $u'\in \mathscr{E}$ sufficiently close to $u$, $N_{\eta}(u')=N_{\eta}(u)$. From the first point (applied $N_{\eta}(u)$ times), we know that for $\delta>0$ small enough, $N_{\eta}(u')\ge N_{\eta}(u)$ whenever $d(u,u')<\delta$. To prove that $N_{\eta}(u')\le N_{\eta}(u)$ holds as well when $\delta$ is sufficiently small, we use an argument by contradiction and we consider a sequence $(u_n)_{n\ge 1}$ of elements in $\mathscr{E}$ such that $d(u,u_n)\rightarrow 0$ and $N_{\eta}(u_n)\ge N_{\eta}(u)+1$. Consider $N_\eta(u)+1$ distinct excursion time intervals $(i^{(n)}_{j,-},i^{(n)}_{j,+})$, $1\le j\le N_\eta(u)+1$, of $u_n$ above $a$ such that $i^{(n)}_{j,+}-i^{(n)}_{j,-}>\eta$. We can write the corresponding excursions $e_a^{(t^{(n)}_j)}(u_n)$ for some $t_j^{(n)}$'s. Moreover, we may take $t_j^{(n)}$ such that $|i^{(n)}_{j,+}-t_j^{(n)}|>\eta/2$ and $|i^{(n)}_{j,-} -t_j^{(n)}|>\eta/2$. Since $|R(u)-R(u_n)|\rightarrow 0$, we can assume (up to some extraction) that when $n$ goes to infinity, $i^{(n)}_{j,+} \rightarrow i_{j,+}$, $i^{(n)}_{j,-}\rightarrow i_{j,-}$ and $t_j^{(n)}\rightarrow t_j\in [0,R(u)]$, for some $i_{j,+},i_{j,-}, t_j \in [0,R(u)]$.  From $u_n\rightarrow u$, we deduce that for all $j$, $y(i_{j,-})=a$ and $y(i_{j,+})=a$. For $n$ large enough, because $i^{(n)}_{j,+}-i^{(n)}_{j,-}>\eta$ and $|i^{(n)}_{j,\pm} -t_j^{(n)}|>\eta/2$, we have $e_a^{(t_j^{(n)})}(u_n)=e_a^{(t_j)}(u_n)$. Now consider $e_a^{(t_j)}(u)$. From the two previous points, $e_a^{(t_j)}(u_n) \rightarrow e_a^{(t_j)}(u)$. For any time  $s\in (i_-,i_+)$, we have $y(s)>a$ (otherwise $a$ would be a local minimum of $y$). Hence $(i_{j,-},i_{j,+})$ is an excursion time interval for $u$ and $i_{j,+}-i_{j,-}>\eta$. Therefore we constructed $N_\eta(u)+1$ distinct excursion time intervals above $a$ for $u$, which gives the desired contradiction.
    \item[-] \underline{Step 2:} Suppose for example that $z_i^{a,+}>0$. Take $\delta<\frac{z_i^{a,+}}{6}$ and $\eta=\eta(\delta)>0$ some modulus of uniform continuity for $u$ with respect to $\delta$. We can assume in order to apply Step 1 that $\eta$ is such that $u$ has no excursion above $a$ satisfying $|i_+-i_-|=\eta$. We look at the $N:=N_{\eta}(u)$ excursions $e_1,\ldots, e_N$ of $u$ above $a$ (ranked by decreasing order of the absolute value of their sizes) such that $|i_+-i_-|>\eta$, and denote their sizes by $z_1,\ldots,z_N$. Observe that the first $i$ excursions among these are the excursions $e_1^{a,+}, \ldots, e_i^{a,+}$. Indeed, if $|i_+-i_-|\le \eta$, then by uniform continuity,
\[
|u(i_+)-u(i_-)| 
\le \delta < z_i^{a,+}.
\]

Let $\varepsilon'= \frac{1}{2} \left(\min_{1\le k\le N-1} |z_{k+1}-z_k| \wedge z_i \right)$ (this is positive since all the sizes are assumed to be distinct in $\mathscr{E}$). Take times $t_1,\ldots,t_N$ in the excursion time intervals of $e_1,\ldots, e_N$. Thanks to Step 1 and the first point of the proof (applied $N$ times), there exists $\delta'>0$ such that for $d(u,u')<\delta'$, if we denote by $(i_-'^{(t_k)},i_+'^{(t_k)})$ the excursion time interval of $e'^{(t_k)}_a, 1\le k\le N$, then
\begin{enumerate}
    \item[(i)] $N_{\eta}(u')=N$,
    \item[(ii)] the excursions $e'^{(t_k)}_a, 1\le k\le N,$ are distinct,
    \item[(iii)] $\forall \, 1\le k\le N, \quad  |i_+'^{(t_k)}-i_-'^{(t_k)}|>\eta$,
    \item[(iv)] $\forall \, 1\le k\le N, \quad |z^{a,+}_k-\Delta e'^{(t_k)}_a|\le \varepsilon'.$
\end{enumerate}
An easy calculation shows that by our choice of $\varepsilon'$ and  (iv), the $\Delta e'^{(t_k)}_a, 1\le k\le N$, are ranked in decreasing order, and that 
\begin{equation}
\forall\, 1\le k\le i, \quad \Delta e'^{(t_k)}_a>\frac{z_i^{a,+}}{2}. \label{eq:lower-bound}
\end{equation}
In addition, by (i), (ii) and (iii), the $e'^{(t_k)}_a, 1\le k\le N,$ are the excursions of $u'$ above $a$ satisfying $|i_+-i_-|>\eta$.

Now set $\delta''=\min(\delta,\delta')$ and assume that $d(u,u')< \delta''$. Then for all $1\le k\le i$, $e'^{(t_k)}_a = e_k^{a,+}(u')$. Indeed, if $(i_-,i_+)$ is an excursion time interval of $u'$ such that $|i_+-i_-|\le \eta$, then 
\[
|u'(i_+)-u'(i_-)| 
\le |u'(i_+)-u(i_+)|+ |u(i_+)-u(i_-)|+ |u(i_-)-u'(i_-)| 
\le 3\delta,
\]
and so in particular $|u'(i_+)-u'(i_-)| < \frac{z_i^{a,+}}{2}$. This proves that the first $i$ excursions of $u'$ are among the $N$ previous excursions satisfying $|i_+-i_-|>\eta$. Since these are ranked in decreasing order, necessarily $e'^{(t_k)}_a = e_k^{a,+}(u')$ for all $1\le k\le i$, which concludes the proof.
\end{itemize}
\end{itemize}
Putting these three points together, we proved that $e_i^{a,+}$ is continuous on $\mathscr{E}$ which has full probability under $\gamma_z$, hence Proposition \ref{prop:continuity excursions}. 
\end{proof}


\section{Markovian properties} \label{Sec:Markovian}
In this section, we are interested in Markovian properties of excursions cut at horizontal levels. Time will therefore be indexed by the height $a$ of the cutting. 


\subsection{The branching property for excursions in $\Hb$} \label{section: branching property}
Consider an excursion under the measure $\gamma_z$. Then cutting it at some height $a>0$ yields a family of excursions above $a$ as defined in Section \ref{excursion level}. Our aim is to show that conditionally on what happens below $a$, these are independent and distributed according to the measures $\gamma_z$, where $z$ is the size of the corresponding excursion. We shall first consider the case when the original excursion is taken under the It\^o's measure $\n_+$ in $\Hb$, and then transfer the property to $\gamma_z$ by the previous disintegration result (\ref{disintegration}).

Let $\mathcal{G}^0_a$ be the $\sigma$--field containing all the information of the trajectory below level $a$ and $\mathcal{G}_a$ be the completion of $\mathcal{G}^0_a$ with the $\n_+$--negligible sets.  In other words, the $\sigma$--field $\mathcal{G}^0_a$ is generated by the trajectory $u$ once you cut out the excursions above $a$, and close up the time gaps. A formal definition of this process is the process $u$ indexed by the generalized inverse of $t\mapsto \int_0^t \mathds{1}_{\{u(s) \leq a \}}\mathrm{d}s$. 

\begin{figure} 
\begin{center}
\includegraphics[scale=0.75]{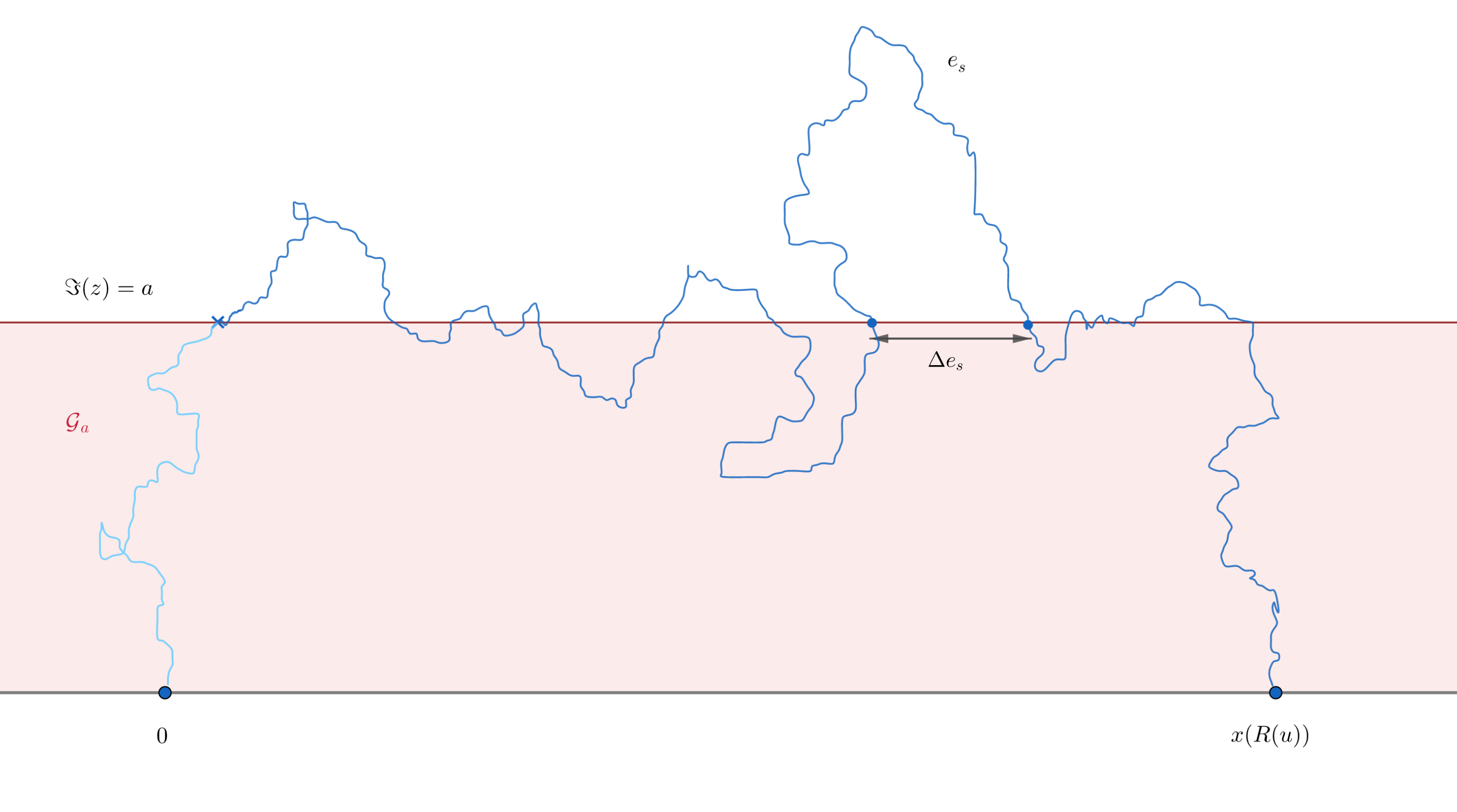}
\end{center}
\caption{The excursion process above $a$.}
\label{PPP fig}
\end{figure}

Recall from Section \ref{Sec:metric space} that $z^{a,+}_1, z^{a,+}_2, \ldots$ are the sizes of the excursions above $a$, ranked in decreasing order of their absolute value, and $e^{a,+}_1 , e^{a,+}_2, \ldots$ are the corresponding excursions.    

\begin{Prop} \label{branching theorem}(Branching property for excursions in $\Hb$ under $\n_+$)

For any $A\in \Gcal_a$, and for all nonnegative measurable functions $G_1,\ldots,G_k:U^+\rightarrow \R_+$, $k\ge 1$,
\begin{equation}  \label{branching}
\n_+\left(\mathds{1}_{\{T_a<\infty\}}\mathds{1}_A \prod_{i=1}^k G_i(e^{a,+}_i)\right) = \n_+\left(\mathds{1}_{\{T_a<\infty\}}\mathds{1}_A \prod_{i=1}^k \gamma_{z^{a,+}_i}[G_i]\right).
\end{equation}
\end{Prop}

\begin{proof} 
Lemma \ref{Markov under n} ensures that on the event $\{T_a <\infty\}$, the trajectory $u$ after time $T_a$ has the law of a  killed Brownian motion. Excursion theory tells us that given the excursions below $a$, the excursions above $a$ form a Poisson point process on $U^+$ with intensity $L \, \n_+({\rm d} u)$, where $L$ is the total local time at level $a$, see Figure \ref{PPP fig}. Finally, conditionally on the sizes $(z_i^{a,+})_{i\ge 1}$ of the excursions above $a$, these excursions are independent with law $\gamma_{z_i^{a,+}}$. We deduce the proposition since the $\sigma$-field $\mathcal{G}_a$ is generated by $\mathcal{F}_{T_a}$, the excursions below $a$, and the sizes $(z_i^{a,+})_{i\ge 1}$.
\end{proof}

\medskip

We can now transfer this property to the probability measures $\gamma_z$.

\begin{Prop} \label{branching gamma}(Branching property for excursions in $\Hb$ under $\gamma_z$)

Let $z\in\R\setminus \{0\}$. For any $A\in \Gcal_a$, and for all nonnegative measurable functions $G_1,\ldots,G_k:U^+\rightarrow \R_+$, $k\ge 1$,
\[\gamma_z\left(\mathds{1}_{\{T_a<\infty\}}\mathds{1}_A \prod_{i=1}^k G_i(e^{a,+}_i)\right) = \gamma_z\left(\mathds{1}_{\{T_a<\infty\}}\mathds{1}_A \prod_{i=1}^k \gamma_{z^{a,+}_i}[G_i]\right).\]
\end{Prop}
\begin{proof} It suffices to prove the proposition for bounded continuous functions $G_1,\ldots,G_k:U^+\rightarrow \R_+$, $k\ge 1$. Take a nonnegative measurable function $f:\R\rightarrow \R_+$ and a bounded continuous function $h:U^+ \rightarrow \R_+$ which is $\Gcal_a-$measurable. Observe that $x(R(u))$ is $\Gcal_a-$measurable as a function of $u$. From Proposition \ref{branching theorem}, we know that
\[
\n_+\left(\mathds{1}_{\{T_a<\infty\}}h(u) f(x(R(u))) \prod_{i=1}^k G_i(e^{a,+}_i)\right) = \n_+\left(\mathds{1}_{\{T_a<\infty\}}h(u) f(x(R(u))) \prod_{i=1}^k \gamma_{z^{a,+}_i}[G_i]\right).
\]
Thanks to the disintegration formula (\ref{disintegration}), we can split $\n_+$ over the size:
\[
\int_{\R} \frac{\mathrm{d}z}{2\pi z^2}\, f(z)\,\gamma_z\left(\mathds{1}_{\{T_a<\infty\}}h \prod_{i=1}^k G_i(e^{a,+}_i)\right) = \int_{\R} \frac{\mathrm{d}z}{2\pi z^2}\, f(z)\,\gamma_z\left(\mathds{1}_{\{T_a<\infty\}} h \prod_{i=1}^k \gamma_{z^{a,+}_i}[G_i]\right).
\]
Since this holds for any $f$, it entails for Lebesgue-almost every $z\in\R$,
\begin{equation} \label{eq. gamma}
\gamma_z\left(\mathds{1}_{\{T_a<\infty\}} h \prod_{i=1}^k G_i(e^{a,+}_i)\right)
=
\gamma_z\left(\mathds{1}_{\{T_a<\infty\}} h \prod_{i=1}^k \gamma_{z^{a,+}_i}[G_i]\right).
\end{equation}
To prove that this holds for all $z$, we need a continuity argument. We first treat the case $z=1$. Using the scaling property \ref{l:scaling} of the measures $\gamma_z$, for $z>0$ the left-hand side of ($\ref{eq. gamma}$) is 
\[
 \gamma_1 \left(\mathds{1}_{ \{T_{a/z}<\infty\}} h (u^{(z)}) \prod_{i=1}^k G_i(e^{a,+}_i(u^{(z)}))\right)
\]
where we recall from Lemma \ref{lem:continuity} that $u^{(z)}=zu(\cdot/z^2)$. The right-hand side term, on the other hand, is 
\[
\gamma_z\left(\mathds{1}_{\{T_a<\infty\}}h \prod_{i=1}^k \gamma_{z^{a,+}_i}[G_i]\right)
=
\gamma_1 \left(\mathds{1}_{\{ T_{a/z}<\infty \}}h(u^{(z)}) \prod_{i=1}^k \gamma_{z^{a,+}_i(u^{(z)})}[G_i]\right),
\]
and so ($\ref{eq. gamma}$) translates into 
\begin{equation} \label{eq. scaled}
\gamma_1 \left(\mathds{1}_{ \{ T_{a/z}<\infty \}}h(u^{(z)}) \prod_{i=1}^k G_i(e^{a,+}_i(u^{(z)}))\right)
=
\gamma_1 \left(\mathds{1}_{ \{ T_{a/z}<\infty \}}h(u^{(z)}) \prod_{i=1}^k \gamma_{z^{a,+}_i(u^{(z)})}[G_i]\right),    
\end{equation}
for Lebesgue-almost every $z>0$. In particular this is true for a dense set of $z$. Taking $z\searrow 1$ along some decreasing sequence, we first get 
that $u^{(z)}\rightarrow u$ by Lemma \ref{lem:continuity} and $T_{a/z}\rightarrow T_a$ by left-continuity of the stopping times. In addition, for all $1\le i\le k$, $z^{a,+}_i(u^{(z)})\rightarrow z^{a,+}_i(u)$ $\gamma_1$-almost surely because $z\rightarrow z^{a,+}_i(u^{(z)})=\Delta e_i^{a,+}(u^{(z)})$ is a continuous function (outside a negligible set) by Lemmas \ref{lem:continuity delta}, \ref{lem:continuity} and Proposition \ref{prop:continuity excursions}. Finally, by continuity of $z\mapsto \gamma_z$ (Lemma \ref{prop: continuity gamma}), for all $1\le i\le k$, $\gamma_{z^{a,+}_i(u^{(z)})}[G_i] \rightarrow \gamma_{z^{a,+}_i}[G_i]$. Applying the dominated convergence theorem to both sides of equation ($\ref{eq. scaled}$) triggers
\[
\gamma_1\left(\mathds{1}_{\{T_a<\infty\}} h \prod_{i=1}^k G_i(e^{a,+}_i)\right) = \gamma_1\left(\mathds{1}_{\{T_a<\infty\}}h \prod_{i=1}^k \gamma_{z^{a,+}_i}[G_i]\right).
\]
and concludes the proof of Proposition \ref{branching gamma} for $z=1$. The general case follows by scaling.
\end{proof}


\subsection{The locally largest evolution} \label{Sec:locally-largest}
Recall that Proposition \ref{locally largest prop} gives a canonical choice of excursion at level $a>0$, which is the locally largest excursion $e_a^{(\tb)}$. One may wonder whether the locally largest fragment $\Xi(a)= \Delta e_a^{(\tb)}$ still exhibits some kind of Markovian behavior. The following theorem answers this question. 

\begin{Thm} \label{thm:locally largest}
Let $z>0$. Under $\gamma_z$, $(\Xi(a))_{0\le a< \Im(\zb)}$ is distributed as the positive self-similar Markov process $(Z_a)_{0\le a<\zeta}$ with index $1$ starting from $z$ whose Lamperti representation is 
\[Z_a = z\exp(\xi(\tau(z^{-1}a))), \]
where $\xi$ is the L\'evy process with Laplace exponent $\Psi(q) := \ln \gamma_z [\mathrm{e}^{q\xi(1)}]$ given by
\begin{equation} \label{Lapl loc largest}
    \Psi(q) = -\frac4\pi q + \frac2\pi \int_{y>-\ln(2)} \left(\mathrm{e}^{q y}-1-q(\mathrm{e}^y-1)\right) \frac{\mathrm{e}^{-y}\mathrm{d}y}{(\mathrm{e}^y-1)^2}, \quad q<3,
\end{equation}
$\tau$ is the time change 
\[\tau(a) = \inf\left\{s\ge 0, \; \int_{0}^s \mathrm{e}^{\xi(u)}\mathrm{d}u > a\right\},\]
and $\zeta = \inf\{a\ge 0, \; Z_a=0\}$.
\end{Thm}

Recall the notation \eqref{def:uleft}-\eqref{def:Tright}. We set
\[
u^{t, \rightleftharpoons}_a := \left((u^{t,\leftarrow}(s+T_a^{t,\leftarrow}))_{s\ge 0},(u^{t,\rightarrow}(s+T_a^{t,\rightarrow}))_{s\ge 0} \right) - \left(u^{t,\leftarrow}(T_a^{t,\leftarrow}),u^{t,\leftarrow}(T_a^{t,\leftarrow})\right),
\]
with the convention that $u^{t, \rightleftharpoons}_a$ is a cemetery function if $y(t) < a$. 

We shall use the following lemma. Note that the lemma does not disintegrate the law of $u^{\tb, \rightleftharpoons}_a$ on the measures $\gamma_z$, and one has to be careful not to  confuse the $z$ appearing in the integral with the value of $x(R(u))$ (the reader should keep track of $x(R(u))$ in the proof). 
\begin{Lem} \label{lemma:locally largest}
 Let $(X,Y)$ and $(X',Y')$ be under $\Pb$ two independent planar Brownian motions starting from the origin, and for $b\le 0$, $T_b$ and $T'_b$ their respective hitting times of $\{\Im(z)=b\}$, with $\widetilde{T}_{b}$ and $\widetilde{T}'_{b}$ denoting the hitting times of $\{\Im(z)<b\}$. For $\alpha \ge a\ge 0$, and $z\in\R$ we set
\[
\mathcal{E}_{\alpha ,a,z} := \left\{\left|z+(X'_{T'_{b}}-X_{T_{b}})\right| \ge  \left| (X'_{\widetilde{T}'_b}-X_{\widetilde{T}_b}) - (X'_{T'_{b}}-X_{T_{b}}) \right|,\, \forall\, b\in [-\alpha ,-\alpha +a] \right\}.
\]
Then, for any nonnegative measurable function $H$, 
\[
\n_+[H(u^{\tb, \rightleftharpoons}_a)\mathds{1}_{\{a<\Im(\zb)\}}] 
= \int_z {\mathrm{d} z \over 2 \pi z^2} h(-a,z),
\]
where $h$ is :
\[
h(-a,z) := \Eb \left[  H\left( (X_s,Y_s)_{s \in [0,T_{-a}]},  (z+X'_s,Y'_s)_{s \in [0,T'_{-a}]}\right),\, \mathcal{E}_{a,a,z}  \right].
\]
\end{Lem}

{\bf Remark}. Observe that the process $(\Xi(a'),\, a' \le a)$ is measurable with respect to $u^{\tb, \rightleftharpoons}_a$. We denote by $\D$ the space of c\`adl\`ag real-valued paths with finite lifetime, endowed with the local Skorokhod topology. It results from the lemma that for any nonnegative measurable function on $\D$, 
\[
\n_+[H(\Xi(b),\, b \in [0,a])\mathds{1}_{\{a<\Im(\zb)\}}] 
 = \int_z {\mathrm{d} z \over 2 \pi z^2} h(-a,z),
\]
where $h$ is now :
\[
h(-a,z) := \Eb \left[  H\left(z+X'_{T'_{-a+b}}-X_{T_{-a+b}},\, b\in [0,a]\right),\, \mathcal{E}_{a,a,z}  \right].
\]

\begin{proof}
Integrating over the duration of the excursion $e_a^{(\tb)}$, we see that for $\n_+-$almost every $ u\in U^+ $, 
\[
H(u^{\tb, \rightleftharpoons}_a) \mathds{1}_{\{a<\Im(\zb)\}} = \int_{0}^{R(u)} H(u^{t, \rightleftharpoons}_a) \mathds{1}_{\{y(t) > a,\, e_a^{(\tb)}=e_a^{(t)}\}} {1\over R(e_a^{(t)})} \mathrm{d} t.
\]

\noindent With Bismut's description of $\n_+$ (Proposition \ref{Bismut}, cf. Figure \ref{Bismut fig}), we get
\[
    \n_+[H(u^{\tb, \rightleftharpoons}_a)\mathds{1}_{\{a<\Im(\zb)\}}] 
     =  
    \int_{\alpha  > a} \mathrm{d}\alpha\, \Eb \left[ \frac{ H\left( \left((\overline{X},\overline{Y}),(\overline{X'},\overline{Y'})  \right)\right)}{T'_{-\alpha +a}+T_{-\alpha +a}},\, \mathcal{E}_{\alpha ,a,0}  \right],
\]

\noindent where
\begin{align*}
(\overline{X},\overline{Y})
&= 
((X,Y)(s+T_{-\alpha +a}))_{s\in [0,T_{-\alpha }-T_{-\alpha +a}]}- (X,Y)(T_{-\alpha +a}),
\\
(\overline{X'},\overline{Y'}) 
&= 
((X',Y')(s+T'_{-\alpha +a}))_{s\in [0,T'_{-\alpha }-T'_{-\alpha +a}]}
- (X,Y)(T_{-\alpha +a}),
\end{align*}

\noindent  with the notation $(X,Y)(s)=(X_s,Y_s)$. By the strong Markov property at times $T_{-\alpha +a}$ and $T'_{- \alpha +a}$, the former integral can be expressed as
\[
\int_{\alpha  > a} \mathrm{d} \alpha \, \Eb \left[ h\left(-a,X'_{T'_{-\alpha +a}}-X_{T_{-\alpha +a}}\right) \frac1{ T'_{-\alpha +a}+T_{-\alpha +a} }\right],
\]

\noindent for $h$ defined as
\[
h(-a,z) :=  \Eb \left[  H\left( (X_s,Y_s)_{s \in [0,T_{-a}]},  (z+X'_s,Y'_s)_{s \in [0,T'_{-a}]}  \right),\, \mathcal{E}_{a,a,z}  \right].
\]

\noindent See Figure \ref{Bismut fig}. By a change of variables, the former integral is
\[
\int_{\alpha \ge 0} \mathrm{d} \alpha \,  \Eb \left[ h(-a,X'_{T'_{-\alpha }}-X_{T_{-\alpha }}) \frac1{ T'_{-\alpha }+T_{-\alpha } }\right].
\]
Therefore, we proved that
\[
\n_+[H(u^{\tb, \rightleftharpoons}_a )\mathds{1}_{\{a<\Im(\zb)\}}] 
=
\int_{\alpha  \ge 0} \mathrm{d}  \alpha \,  \Eb \left[ h(-a,X'_{T'_{-\alpha }}-X_{T_{-\alpha }}) \frac1{ T'_{-\alpha }+T_{- \alpha } }\right].
\]
On the other hand, using again Bismut's decomposition of $\n_+$, we see that (actually for any $h$),
\begin{align*}
\n_+[h(-a, x(R(u)))] & = 
\n_+\left(\int_{0}^{R(u)} h(-a,x(R(u)) {1\over R(u)} \mathrm{d} t\right)
\\
&=
\int_{\alpha  \ge 0} \mathrm{d} \alpha \,  \Eb \left[  h\left(-a, X'_{T'_{-\alpha }}-X_{T_{-\alpha }} \right) \frac{1}{ T'_{-\alpha }+T_{-\alpha } } \right].
\end{align*}

\noindent Comparing the last two equations, we proved that
\[
\n_+[H(u^{\tb, \rightleftharpoons}_a)\mathds{1}_{\{a<\Im(\zb)\}}] 
=  
\n_+[h(-a, x(R(u)))] = \int_z {\mathrm{d} z \over 2 \pi z^2} h(-a,z),
\]
by Proposition \ref{disintegration prop}. 
\end{proof}
\medskip

We now come to the proof of Theorem \ref{thm:locally largest}. We closely follow the strategy of Le Gall and Riera in \cite{LG-Riera}.

\begin{proof} Let $H$ be a nonnegative bounded continuous function on $\D$. 
From the previous Lemma \ref{lemma:locally largest}, or rather from the Remark following its statement, we know that
\[
\n_+[H(\Xi(b),\, b \in [0,a])\mathds{1}_{\{a<\Im(\zb)\}}] 
= \int_z {\mathrm{d} z \over 2 \pi z^2} h(-a,z),
\]
where $h$ is:
\[
h(-a,z) = \Eb \left[  H\left(z+X'_{T'_{-a+b}}-X_{T_{-a+b}},\, b\in [0,a]\right),\, \mathcal{E}_{a,a,z}  \right].
\]

\noindent Notice that, in the notation of Lemma \ref{lemma:locally largest}, $b\mapsto X'_{\widetilde{T}'_{-b}}-X_{\widetilde{T}_{-b}}$ is a (c\`adl\`ag) symmetric Cauchy process of Laplace exponent $\psi(\lambda) = -2 |\lambda|$ (for example, use that it is a L\'evy process and Proposition 3.11 of \cite{RY}, Chap. III). Denote by $\eta_b$ the double of the Cauchy process which under $P_z$, starts from $z$, and $\Delta \eta_b$ the jump at time $b$. Write $\hat{\eta}_b=\eta_{(a-b)^-}$ for the time-reversal of $\eta$, $\Delta \hat \eta_b$ being the jump of $\hat \eta$ at time $b$. Then by definition of $h$,
\[
h(-a,z) = E_z\left[  H( \hat{\eta}_{b},\, b\in [0,a])\, \mathds{1}_{\{\forall \, b  \in [0,a], \; |\hat{\eta}_b| \geq  |\Delta \hat{\eta}_{b}| \}} \right].
\]
Now we want to reverse time in the function $h$. Conditioning on $\eta_a$,  
\begin{align*}
&h(-a,z) \\ 
&=\frac1\pi\int_\R  {2a \mathrm{d} x \over (2a)^2 + (x-z)^2} 
E_z[  H( \hat{\eta}_{b},\, b\in [0,a])\, \mathds{1}_{\{\forall \, b  \in [0,a], \; |\hat{\eta}_{b}| \ge |\Delta \hat{\eta}_{b}| \}} | \eta_a = x ].
\end{align*}

\noindent By Corollary 3, Chap. II of \cite{Ber4}:
\begin{align*}
& E_z[  H( \hat{\eta}_{b},\, b\in [0,a])\, \mathds{1}_{\{\forall \, b  \in [0,a], \; |\hat{\eta}_{b}| \ge |\Delta \hat{\eta}_b | \}} | \eta_a = x ] \\
&=
E_x[  H(\eta_{b},\, b\in [0,a])\, \mathds{1}_{\{\forall \, b  \in [0,a], \; |\eta_b| \ge | \Delta\eta_b | \}} | \eta_a = z ].    
\end{align*}
Indeed, the Cauchy process $\eta$ is symmetric, hence is itself its dual. We obtain 
\begin{align*}
&\int_{\R} {\mathrm{d} z \over 2 \pi z^2} h(-a,z) =\\
&\int_{\R} {\mathrm{d} z \over 2 \pi z^2} \frac1\pi \int_{\R}  {2a \mathrm{d} x \over (2a)^2 + (x-z)^2}
E_x [  H(\eta_{b},\, b\in [0,a])\, \mathds{1}_{\{\forall \, b  \in [0,a], \; |\eta_b| \ge |\Delta \eta_{b}| \}} | \eta_a = z ].
\end{align*}
We can rewrite it as
\[
\int_\R { \mathrm{d} x \over 2 \pi x^2}  \frac1\pi\int_\R  {2a \mathrm{d} z \over (2a)^2 + (x-z)^2}   
E_x \left[ { x^2  \over  z^2}  H(\eta_{b},\, b\in [0,a])\, \mathds{1}_{\{\forall \, b  \in [0,a], \; |\eta_b| \ge |\Delta \eta_{b}| \}} \bigg| \eta_a = z \right],
\]
which is
\[
 \int_\R { \mathrm{d} x \over 2 \pi x^2} E_x \left[ { x^2  \over  (\eta_a)^2}  H(\eta_{b},\, b\in [0,a])\, \mathds{1}_{\{\forall \, b  \in [0,a], \; |\eta_b| \ge |\Delta \eta_{b}| \}} \right].
\]

\noindent Now this gives the law of $\Xi$ under the disintegration measures $\gamma_x$. Indeed, take instead of $H$ some nonnegative measurable function $f$ of the initial size $\Xi(0)$, multiplied by $H$. Then using the above expression, we find that
\begin{align*}
&\n_+[f(\Xi(0))H(\Xi(b),\, b \in[0,a])\mathds{1}_{\{a<\Im(\zb)\}}] \\
&=
\int_\R { \mathrm{d} x \over 2 \pi x^2} f(x) E_x \left[ { x^2  \over  (\eta_a)^2}  H(\eta_{b},\, b\in [0,a])\, \mathds{1}_{\{\forall \, b  \in [0,a], \; |\eta_b| \ge |\Delta \eta_{b}| \}} \right].
\end{align*}
Hence for Lebesgue-almost every $x\in\R$, 
\begin{equation} \label{eq:cauchy}
\gamma_x[H(\Xi(b),\, b \in [0,a] )\mathds{1}_{\{a<\Im(\zb)\}}]
=
E_x \left[ { x^2  \over  (\eta_a)^2}  H(\eta_{b},\, b\in [0,a])\, \mathds{1}_{\{\forall \, b  \in [0,a], \; |\eta_b| \ge |\Delta \eta_{b}|\}} \right],
\end{equation}
and by continuity this must hold for all $x\in\R$. Indeed, by scaling, the left-hand side of equation (\ref{eq:cauchy}) is 
\[\gamma_x[H(\Xi(b),\, b \in [0,a])\mathds{1}_{\{a<\Im(\zb)\}}]
=
\gamma_1 [H(x\Xi(x^{-1}b),\, b \in [0,a])\mathds{1}_{\{a<x\Im(\zb)\}}].\]
The right-hand term can be put in the same form by using the scale invariance of the Cauchy process. Since (\ref{eq:cauchy}) holds for almost every $x$, it must hold on a dense set of $x>0$, and we may take $x\nearrow 1$ along a sequence. By dominated convergence, we get
\[
\gamma_1 [H(\Xi(b),\, b \in [0,a])\mathds{1}_{\{a<\Im(\zb)\}}]
=
E_1 \left[ { 1  \over  (\eta_a)^2}  H(\eta_{b},\, b\in [0,a])\, \mathds{1}_{\{\forall \, b  \in [0,a], \; |\eta_b| \ge |\Delta \eta_{b}| \}} \right],\]
and this proves that equation (\ref{eq:cauchy}) holds for $x=1$. The general case $x\in\R$ follows by scaling.

Notice that, almost surely, on the event $\{ \forall \, b  \in [0,a], \; |\eta_b| \ge |\Delta \eta_{b}| \}$, if $\eta_0>0$, then $\eta_b$ is positive for all $b\in[0,a]$. We know from \cite{CC} that a symmetric Cauchy process starting from $x>0$ killed when entering the negative half-line can be written using its Lamperti representation as $x \mathrm{e}^{\xi^0 (\tau^0(a))}$ where
 \[
 \tau^0(a) := \int_0^a \frac{\mathrm{d} s}{\eta_s} = \inf\left\{s\ge0, \; \int_0^s x\mathrm{e}^{\xi^0(u)} \mathrm{d}u \ge a\right\},
 \]
and $(\xi^0(a),\, a\ge 0)$ is under $P$ 
a  L\'evy process killed at an exponential time of parameter $\frac{2}{\pi}$, starting from $0$ with Laplace exponent
\begin{equation} \label{Psi}
\Psi^0(q) = \frac{2}{\pi} \int_{\R} (\mathrm{e}^{q y} -1 - q (\mathrm{e}^{y}-1) \mathds{1}_{|\mathrm{e}^y -1|<1} ) \mathrm{e}^y (\mathrm{e}^y-1)^{-2} \mathrm{d} y - \frac2\pi, \quad -1<q<1.  
\end{equation}

Let $\Delta \xi^0_b$ denote the jump of $\xi^0$ at time $b$, i.e. $\Delta \xi^0_b:= \xi^0_b - \xi^0_{b^-}$. The following lemma is the analog of Lemma 17 in \cite{LG-Riera}.
\begin{Lem}
For every $a\ge 0$, set
\[
M_a = \mathrm{e}^{-2\xi^0_a}\mathds{1}_{\{\forall\, b \in [0,a], \;  \Delta \xi^0_b > -\ln(2)\}}.
\]
Then $(M_a)_{a\ge 0}$ is a martingale with respect to the canonical filtration of the process $\xi^0$. Under the tilted probability measure $\mathrm{e}^{-2\xi^0_a}\mathds{1}_{\{\forall\, b \in [0,a], \;  \Delta \xi^0_b > -\ln(2)\}}\cdot P$, the process $(\xi^0(b))_{b\in [0,a]}$ is a L\'evy process with Laplace exponent $\Psi$ introduced in \eqref{Lapl loc largest} in Theorem \ref{thm:locally largest}. 
\end{Lem}
\begin{proof}
We compute
\[
E[\mathrm{e}^{(q-2)\xi_a^0}\mathds{1}_{\{\forall\, b \in [0,a], \;  \Delta \xi_b^0 > -\ln(2)\}} ].
\]

Indeed, that $(M_a)_{a\ge 0}$ is a martingale will come from the fact that $\xi^0$ is a L\'evy process and that the expectation above is $1$ when $q=0$.  
To compute this expectation, we decompose $\xi^0$ into its small and large jumps parts:
\[
\xi_a^0 = \xi_a' + \xi_a'',
\]
where $\xi_a'' = \sum_{0\le b\le a} \Delta \xi_b^0 \mathds{1}_{\Delta\xi_b^0\le -\ln(2)}$. Notice that and $\xi'$ and $\xi''$ are independent. Then by independence, the above expectation is
\begin{align}
E[\mathrm{e}^{(q-2)\xi_a^0}\mathds{1}_{\{\forall\, b \in [0,a], \;  \Delta \xi_b^0 > -\ln(2)\}} ] &= E[ \mathds{1}_{\{\xi_a''=0\}} \,\mathrm{e}^{(q-2)\xi_a'}] \nonumber \\
&= P(\xi_a''=0) E[\mathrm{e}^{(q-2)\xi_a'}]. \label{newPsi}
\end{align}
Thus, we need to compute the Laplace exponents of $\xi'$ and $\xi''$ (under $P$), that we denote respectively by $\Psi'$ and $\Psi''$. Because $\xi''$ is the pure-jump process given by the jumps of $\xi^0$ smaller than $-\ln(2)$, its Laplace exponent is given by the L\'evy measure of $\xi^0$ restricted to $(-\infty,-\ln(2)]$, namely
\begin{equation} \label{Psi''}
\Psi''(q) = \frac{2}{\pi} \int_{y\le -\ln(2)} \left(\mathrm{e}^{q y}-1\right) \frac{\mathrm{e}^y}{(\mathrm{e}^y-1)^2} \mathrm{d}y.
\end{equation}
It results from the independence of $\xi'$ and $\xi''$ that the Laplace exponent of $\xi'$ is $\Psi' = \Psi^0-\Psi''$, hence by equations (\ref{Psi}) and (\ref{Psi''}), for all $-1<q<1$,
\begin{multline} \label{Psi'}
\Psi'(q) = \frac{2}{\pi} \int_{y>-\ln(2)} (\mathrm{e}^{q y} -1 - q (\mathrm{e}^{y}-1) \mathds{1}_{\{|\mathrm{e}^y -1|<1\}} ) \mathrm{e}^y (\mathrm{e}^y-1)^{-2} \mathrm{d} y \\
-\frac{2}{\pi} q \int_{y\le -\ln(2)} (\mathrm{e}^{ y}-1) \mathds{1}_{|\mathrm{e}^y-1|<1}\, \mathrm{e}^y(\mathrm{e}^y-1)^{-2} \mathrm{d}y - \frac2\pi.
\end{multline}
The middle term in this expression (\ref{Psi'}) is
\begin{align*}
\frac{2}{\pi} q \int_{y\le -\ln(2)} (\mathrm{e}^{ y}-1)  \mathrm{e}^y(\mathrm{e}^y-1)^{-2} \mathrm{d}y &= - \frac{2}{\pi} q \int_{y\le -\ln(2)}  \frac{\mathrm{e}^y}{1-\mathrm{e}^y} \mathrm{d}y   \\
&= -\frac2\pi q \int_0^{1/2} \frac{\mathrm{d}x}{1-x} \\
&= -\frac2\pi q\ln(2).
\end{align*}
Hence 
\begin{multline} \label{Psi'_2}
\Psi'(q) = \frac{2}{\pi} \int_{y>-\ln(2)} (\mathrm{e}^{q y} -1 - q (\mathrm{e}^{y}-1) \mathds{1}_{|\mathrm{e}^y -1|<1} ) \mathrm{e}^y (\mathrm{e}^y-1)^{-2} \mathrm{d} y \\
+\frac2\pi q\ln(2) - \frac2\pi.
\end{multline}
This extends analytically to all $q<1$. Let us come back to (\ref{newPsi}). We have for $q<3$
\begin{align*}
E [\mathrm{e}^{(q-2)\xi_a^0}\mathds{1}_{\forall\, b \in [0,a], \;  \Delta \xi^0_b > -\ln(2)} ] &= P(\xi_a''=0) E[\mathrm{e}^{(q-2)\xi_a'}] \\
&= \mathrm{e}^{a\Psi''(\infty)} \mathrm{e}^{a\Psi'(q-2)} \\
&= \exp\left(-\frac2\pi a \int_{y\le -\ln(2)} \frac{\mathrm{e}^y}{(\mathrm{e}^y-1)^2} \mathrm{d}y\right) e^{a\Psi'(q-2)} \\
& = \mathrm{e}^{a(\Psi'(q-2)-\frac{2}{\pi})},
\end{align*}
by a change of variables. 

This essentially concludes the calculation of the new Laplace exponent $\widetilde{\Psi}$ of $\xi^0$ under the tilted measure $e^{-2\xi_a^0}\mathds{1}_{ \{\forall \, b  \in [0,a], \; \Delta \xi_b^0 \ge -\ln(2)\}}\cdot P$, which is simply
\begin{equation} \label{Psi tilde}
\widetilde{\Psi}(q) = \Psi'(q-2)-\frac{2}{\pi}, \quad q<3. 
\end{equation} 

Still we can put it in a L\'evy-Khintchin form. Replacing $q$ by $q-2$ in the integral in (\ref{Psi'_2}), we get 
\begin{align*}
& \int_{y>-\ln(2)} (\mathrm{e}^{-2y}\mathrm{e}^{q y} -1 - (q-2) (\mathrm{e}^{y}-1) \mathds{1}_{|\mathrm{e}^y -1|<1} ) \mathrm{e}^y (\mathrm{e}^y-1)^{-2} \mathrm{d} y \\   
&=  \int_{y>-\ln(2)} (\mathrm{e}^{q y} -\mathrm{e}^{2 y}  - (q-2) (\mathrm{e}^{3y}-\mathrm{e}^{2 y}) \mathds{1}_{|\mathrm{e}^y -1|<1} ) \mathrm{e}^{-y} (\mathrm{e}^y-1)^{-2} \mathrm{d} y \\
&= \int_{y>-\ln(2)} (\mathrm{e}^{q y} -1 - q (\mathrm{e}^{y}-1) \mathds{1}_{|\mathrm{e}^y -1|<1} ) \mathrm{e}^{-y} (\mathrm{e}^y-1)^{-2} \mathrm{d} y \\
& + \int_{y>-\ln(2)} \left[1-\mathrm{e}^{2 y} + \left(q(\mathrm{e}^y-1)- (q-2)(\mathrm{e}^{3y}-\mathrm{e}^{2 y})\right) \mathds{1}_{|\mathrm{e}^y -1|<1} \right] \frac{\mathrm{e}^{-y}} {(\mathrm{e}^y-1)^{2}} \mathrm{d}y.
\end{align*}
After simplifications, we find that the last integral is equal to 
\begin{multline} \label{LK-calc}
\int_{y>-\ln(2)} \left[1-\mathrm{e}^{2 y} + \left(q(\mathrm{e}^y-1)- (q-2)(\mathrm{e}^{3y}-\mathrm{e}^{2 y})\right) \mathds{1}_{|\mathrm{e}^y -1|<1} \right] \frac{\mathrm{e}^{-y}} {(\mathrm{e}^y-1)^{2}} \mathrm{d}y \\
=2+2\ln(2)-q\left(2\ln(2)+\frac32\right).
\end{multline}
From equations (\ref{Psi tilde}), (\ref{Psi'}) and (\ref{LK-calc}), we deduce 
\begin{equation} \label{Lapl indicator}
\widetilde{\Psi}(q) = -\frac2\pi \left(\ln(2)+\frac32\right)q + \frac2\pi \int_{y>-\ln(2)} \left(\mathrm{e}^{q y}-1-q(\mathrm{e}^y-1)\mathds{1}_{|\mathrm{e}^y-1|<1}\right) \frac{\mathrm{e}^{-y}\mathrm{d}y}{(\mathrm{e}^y-1)^2}.    
\end{equation}

\noindent Finally, we can remove the indicator using simple calculations. One finds that 
\[\int_{y>-\ln(2)} (1-\mathrm{e}^y)\frac{\mathrm{e}^{-y}}{(\mathrm{e}^y-1)^2} \mathds{1}_{\{|\mathrm{e}^y-1|\ge 1\}}\mathrm{d}y = \frac12-\ln(2),\]
and therefore 
\[\widetilde{\Psi}(q) = -\frac4\pi q + \frac2\pi \int_{y>-\ln(2)} \left(\mathrm{e}^{q y}-1-q(\mathrm{e}^y-1)\right) \frac{\mathrm{e}^{-y}\mathrm{d}y}{(\mathrm{e}^y-1)^2}, \quad q<3.\]

\noindent Hence we recovered the expression for $\Psi$ in the statement of Theorem \ref{thm:locally largest} and this gives both the martingale property and the law of $\xi^0$ under the change of measure.
\end{proof}

We finish the proof of Theorem \ref{thm:locally largest} with the arguments of \cite{LG-Riera} that we reproduce here to be self-contained. Let $x>0$. Equation \eqref{eq:cauchy} reads
\[
\gamma_x[H(\Xi(b),\, b \in [0,a])\mathds{1}_{\{a<\Im(\zb)\}}]
=
E \left[  M_{\tau^0(a)} H\left(x \exp(\xi^0(\tau^0(b))),\, b\in [0,a]\right) \right].
\]

\noindent The optional stopping theorem implies that for any $c>0$,
\begin{align*}
E &\left[ M_{\tau^0(a)} H\left(x \exp(\xi^0(\tau^0(b))),\, b\in [0,a]\right) \mathds{1}_{\{c>\tau^0(a)\}}\right] \\
&=
E\left[ M_{c} H\left(x \exp(\xi^0(\tau^0(b))),\, b\in [0,a]\right)\mathds{1}_{\{c>\tau^0(a)\}} \right].
\end{align*}

\noindent By the lemma, the right-hand side is, with the notation $\xi$ of the theorem,  
\[
E \left[ H\left(x \exp(\xi(\tau(b))),\, b\in [0,a]\right)\mathds{1}_{\{c>\tau(a)\}} \right].
\]

\noindent Making $c\mapsto \infty$ and using dominated convergence completes the proof.
\end{proof}
\medskip

In addition, in order to study the genealogy of the growth-fragmentation process linked to Brownian excursions in the next section, we need to clarify the behavior of the \emph{offspring} of $\Xi$. By offspring we mean all the excursions that were created at times $a$ when the excursion $e^{(\tb)}_a$ divided into two excursions (\emph{i.e.} at jump times of $\Xi$). We rank these excursions in descending order of the absolute value of their sizes. This way we get a sequence $(z_i,a_i)_{i\ge 1}$ of jump sizes and times for $\Xi$, associated to excursions $e_i,i\ge 1$, of size $z_i$ above $a_i$. 

\begin{Thm} \label{idp}
Let $z\in \R\backslash\{0\}$. Under $\gamma_z$, conditionally on the jump sizes and jump times $(z_i,a_i)_{i\ge 1}$ of $\Xi$, the excursions $e_i, i\ge 1$, are independent and each $e_i$ has law $\gamma_{z_i}$.
\end{Thm}

\begin{proof}
By Lemma \ref{lemma:locally largest}, we know that for all nonnegative measurable function $H$,
\[
\n_+[H(u^{\tb, \rightleftharpoons}_a)\mathds{1}_{\{a<\Im(\zb)\}}] 
=  
 \int_z {\mathrm{d} z \over 2 \pi z^2} h(-a,z),
\]
where $h$ is :
\[
h(-a,z) = \Eb \left[  H\left( (X_s,Y_s)_{s \in [0,T_{-a}]},  (z+X'_s,Y'_s)_{s \in [0,T'_{-a}]}\right),\, \mathcal{E}_{a,a,z}  \right].
\]

\noindent Imagine that $H$ is some functional of the offspring of $\Xi$ below level $a$, say $H\left( u^{\tb, \rightleftharpoons}_a \right) = f_1(e^{(a)}_1)\cdots f_n(e^{(a)}_n)$, where the $e^{(a)}_i$ denote the offspring of $\Xi$ created before $a$, ranked in descending order of the absolute value of their sizes $z^{(a)}_i$, and the $f_i$'s are taken continuous and bounded. For such a function $H$, $h$ is given by 
\[h(-a,z)=\Eb\left[f_1(\varepsilon_1)\cdots f_n(\varepsilon_n), \, \mathcal{E}_{a,a,z} \right], \]
where $\varepsilon_1, \ldots, \varepsilon_n$ are the $n$ largest excursions (before hitting $\{\Im(z)=-a\}$) of $(X,Y)$ and $(X',Y')$ above the past infimum of their imaginary parts. Consider the collection $\{(b,e_b^{\rightleftharpoons}), \, b\in[-a,0]\}$ where $e_b^{\rightleftharpoons}$ is an  excursion of the Brownian motions  $(X,Y)$ or $(X',Y')$ above the past infimum of their imaginary parts when the infimum is equal to $b$ (set $e_b^{\rightleftharpoons}=\delta$ if no such excursion exists). A consequence of L\'evy's Theorem, (Theorem 2.3, Chap. VI of \cite{RY}) is that the collection $\{(b,e_b^{\rightleftharpoons}), \, b\le 0\}$ is a Poisson point process of intensity $2 \mathds{1}_{\R_-}\mathrm{d}b\,  \n_+(\mathrm{d}u)$. Write $z(e)$ for the size of an excursion $e$, i.e. the difference between its endpoint and its starting point. Conditionally on the sizes  $\{(b,z(e_b^{\rightleftharpoons})),\, b\le 0 \}$, the excursions $e_b^{\rightleftharpoons}$ are distributed as independent excursions with law $\gamma_{z(e_b^{\rightleftharpoons})}$. Observe that $\mathcal{E}_{a,a,z}$ is measurable with respect to  $\{(b,z(e_b^{\rightleftharpoons})),\, b \le 0 \}$.  Therefore, conditioning on the sizes of the excursions yields
\[h(-a,z)=\Eb\left[\gamma_{z(\varepsilon_1)}(f_1)\cdots \gamma_{z(\varepsilon_n)}(f_n), \, \mathcal{E}_{a,a,z} \right]. \]
And so using Lemma \ref{lemma:locally largest} backwards, we get
\[\n_+\left[f_1(e^{(a)}_1)\cdots f_n(e^{(a)}_n)\mathds{1}_{\{a<\Im(\zb)\}}\right] = \n_+\left[\gamma_{z^{(a)}_1}(f_1)\cdots \gamma_{z^{(a)}_n}(f_n) \mathds{1}_{\{a<\Im(\zb)\}}\right]. \]
Multiplying by a function of the endpoint $x(R(u))$ and disintegrating over it gives 
\[\gamma_z\left[f_1(e^{(a)}_1)\cdots f_n(e^{(a)}_n)\mathds{1}_{\{a<\Im(\zb)\}}\right] = \gamma_z\left[\gamma_{z^{(a)}_1}(f_1)\cdots \gamma_{z^{(a)}_n}(f_n) \mathds{1}_{\{a<\Im(\zb)\}}\right], \]
for Lebesgue-almost every $z\in \R$. Let us prove that this holds for example when $z=1$. By scaling (Lemma \ref{l:scaling}), for $z>0$ this writes 
\begin{align*}
\gamma_1 &\left[f_1(e^{(a)}_1(u^{(z)}))\cdots f_n(e^{(a)}_n(u^{(z)}))\mathds{1}_{\{a<z\Im(\zb)\}}\right] \\ 
&= 
\gamma_1 \left[\gamma_{z^{(a)}_1(u^{(z)})}(f_1)\cdots \gamma_{z^{(a)}_n(u^{(z)})}(f_n) \mathds{1}_{\{a<z\Im(\zb)\}}\right].
\end{align*}
We then condition on the birth times of these excursions. We can apply Proposition \ref{prop:continuity excursions} at different levels and Lemma \ref{lem:continuity} to prove that $\gamma_1$-almost surely, for all $1\le i\le n$, $e^{(a)}_i(u^{(z)})\underset{z\nearrow 1}{\longrightarrow} e^{(a)}_i(u)$ in $U$. Besides, $z^{(a)}_i(u^{(z)}) = \Delta \left(e_i^{(a)}(u^{(z)})\right)$, so by Lemma \ref{lem:continuity delta}, $z^{(a)}_i(u^{(z)}) \underset{z\nearrow 1}{\longrightarrow} z^{(a)}_i(u)$, and by continuity of $z\mapsto \gamma_z$ (Proposition \ref{prop: continuity gamma}), $\gamma_{z^{(a)}_i(u^{(z)})} \underset{z\nearrow 1}{\longrightarrow} \gamma_{z^{(a)}_i(u)}$ almost surely under $\gamma_1$. An application of the dominated convergence theorem finally gives 
\[
\gamma_1\left[f_1(e^{(a)}_1)\cdots f_n(e^{(a)}_n)\mathds{1}_{\{a<\Im(\zb)\}}\right] = \gamma_1\left[\gamma_{z^{(a)}_1}(f_1)\cdots \gamma_{z^{(a)}_n}(f_n)\mathds{1}_{\{a<\Im(\zb)\}} \right].
\]
 
\noindent The statement follows.
\end{proof}

\begin{figure} 
\begin{center}
\includegraphics[scale=0.7]{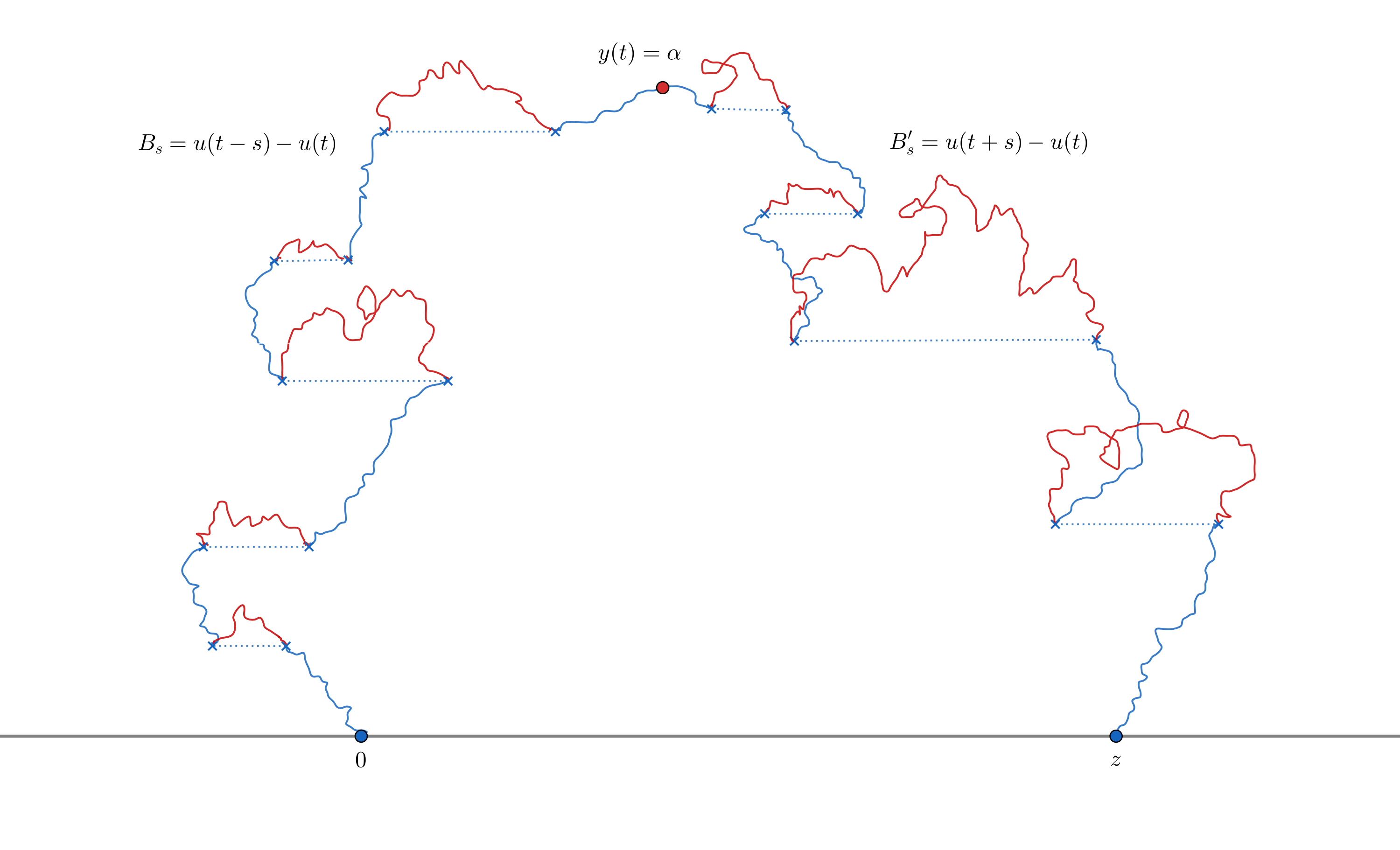}
\end{center}
\caption{Excursions of $B=(X,Y)$ and $B'=(X',Y')$ above their past infimum. The past infimum process is depicted in blue, and by L\'evy's theorem the excursions above it form a Poisson point process represented in red.}
\label{Offspring fig}
\end{figure}


\subsection{A change of measures} \label{s:martingale}
We begin by calling attention to a natural martingale associated to the growth-fragmentation process. 

\begin{Prop} \label{prop:martingale}
Let $z\in \R\backslash\{0\}$. Under $\gamma_z$, the process 
\[{\mathcal M}_a = \mathds{1}_{\{T_a<\infty\}} \sum_{i\geq 1} |\Delta e^{a,+}_i|^2, \quad a\ge 0,\]
is a $(\Gcal_a)_{a\ge 0}-$martingale.
\end{Prop}
\begin{proof}
The branching property \ref{branching theorem} shows that it is enough to prove that $\gamma_z[{\mathcal M}_a]=z^2$ for all $a\ge 0$.

For a  Brownian excursion process $(e_s)_{s>0}$ in the sense of Definition \ref{excursion process}, we use the shorthand $0<s^+\le T$ to denote times $0<s\le T$ such that $e_s \in U^+$. Let $g:\R\rightarrow \R_+$ be a nonnegative measurable function. By the Markov property at time $T_a$, see Lemma \ref{Markov under n}, 
\begin{equation}
\label{martingale}
     \n_+\left( {\mathcal M}_a g(x(R(u)))  \right)
=
       \n_+\left(\mathds{1}_{\{T_a<\infty\}} \, \Eb\left[\sum_{s^+\le L_{T_{-a}}} |\Delta e_s|^2 \, g(X(T_{-a}))\right]\right).
\end{equation}

\noindent By the master formula,
\begin{align*}
\Eb\left[\sum_{s^+\le L_{T_{-a}}} |\Delta e_s|^2 \, g(X(T_{-a}))\right] 
& = \Eb \left[ \int_0^{T_{-a}} \mathrm{d}L_s \int_{-\infty}^{+\infty} \frac{\mathrm{d}z}{2\pi z^2} \, z^2 \, \Eb \left[g(z'+X_{T_{-a}})\right]_{\big|z'=z+X_s}\right] \\
&=
\Eb \left[ \int_0^{T_{-a}} \mathrm{d}L_s \int_{-\infty}^{+\infty} \frac{\mathrm{d}z'}{2\pi }  \, \Eb \left[g(z'+X_{T_{-a}})\right]\right]
\\
& = \Eb\left[ L_{T_{-a}} \frac{1}{2\pi}\int_{-\infty}^{+\infty} g\right],
\end{align*}
since the Lebesgue measure is an invariant measure for the Brownian motion. Finally, the law of the Brownian local time $L_{T_{-a}}$ at the hitting time of $-a$ is known to be exponential with mean $2a$ (see for example Section 4, Chap. VI of \cite{RY}). Hence $\Eb\left[ L_{T_{-a}} \frac{1}{2\pi}\int_{-\infty}^{+\infty} g\right]=2a \times \frac{1}{2\pi}\int_{-\infty}^{+\infty} g$. Coming back to (\ref{martingale}), we get
\[ 
\n_+\left( {\mathcal M}_a g(x(R(u)))  \right)
=
 2a \times \left( \frac{1}{2\pi}\int_{-\infty}^{+\infty} g\right) \n_+\left(T_a<\infty\right). 
\]
But $\n_+\left(T_a<\infty\right) = \n_+\left(\sup(y) \ge a\right) = \frac{1}{2a}$ (see Proposition 3.6, Chapter XII, of \cite{RY}), so finally 
\[
\n_+\left( {\mathcal M}_a g(x(R(u)))  \right)
 = \frac{1}{2\pi}\int_{-\infty}^{+\infty} g.
\]
Disintegrating $\n_+$ over $z$ as in Proposition \ref{disintegration prop} yields
\[
\int_{-\infty}^{+\infty} \frac{\mathrm{d}z}{2\pi z^2} \, g(z) \, \gamma_z[{\mathcal M}_a]  = \frac{1}{2\pi}\int_{-\infty}^{+\infty} g.
\]
This holds for all nonnegative measurable function $g$, and thus for Lebesgue-almost every $z \in \R$, 
\[
\gamma_z[{\mathcal M}_a] = z^2.
\]
Recall the notation $u^{(z)}=zu(\cdot /z^2)$ for $z>0$ from Lemma \ref{lem:continuity}. By scaling, this means for Lebesgue-almost every $z >0$,
\[
\gamma_1 \bigg[ \mathds{1}_{\{z^2 T_{a/z}<\infty\}} \sum_{i\ge 1} |\Delta e^{a,+}_i(u^{(z)})|^2\bigg]
=
z^2,
\]
which yields
\begin{equation}\label{eq:change_dense}
\gamma_1 \bigg[ \mathds{1}_{\{T_{a/z}<\infty\}} \sum_{i\ge 1} |\Delta e^{a,+}_i(u^{(z)})|^2\bigg]
=
z^2.
\end{equation}
Again, this must hold on a dense set of endpoints $z$, and thus taking $z$ according to some sequence, Lemma \ref{lem:continuity} and Proposition \ref{prop:continuity excursions} together with Fatou's lemma imply that $\gamma_1[{\mathcal M}_a] \le 1$. This holds for all $a$, and so by scaling we deduce that for all $z\neq 0$, $\gamma_z[{\mathcal M}_a] \le z^2$. On the other hand, notice that $\Delta e^{a,+}_i(u^{(z)})= z\Delta e^{a/z,+}_i(u)$. By the branching property under $\gamma_1$ (Proposition \ref{branching gamma}), for a $z<1$ such that equation \eqref{eq:change_dense} holds,
\begin{align*} 
1=\gamma_1 \bigg[ \mathds{1}_{\{T_{a/z}<\infty\}} \sum_{i\ge 1} |\Delta e^{a/z,+}_i|^2\bigg]
&=
\gamma_1 \bigg[ \mathds{1}_{\{T_a<\infty\}} \sum_{i\ge 1} \gamma_{\Delta e^{a,+}_i} \bigg( \mathds{1}_{T_{\frac{a}{z}-a}<\infty} \sum_{j\ge 1} |\Delta e^{\frac{a}{z}-a,+}_j|^2\bigg)\bigg] \\
&\le \gamma_1 \bigg[ \mathds{1}_{\{T_{a}<\infty\}} \sum_{i\ge 1} |\Delta e^{a,+}_i|^2\bigg].
\end{align*}
Finally combining the two inequalities, we have $\gamma_1 [{\mathcal M}_a] = 1$, and  $\gamma_z [{\mathcal M}_a] = z^2$ by scaling.
\end{proof}

\bigskip

Associated to this martingale is the change of measures 
\[\frac{\mathrm{d}\mu_z}{\mathrm{d}\gamma_z}\bigg|_{\Gcal_a} = \frac{ {\mathcal M}_a}{z^2}, \quad a\ge 0.\]

We aim at making explicit the law $\mu_z$. 
Following Chap. 5.3 of \cite{Law}, we call $\Hb-$excursion a process in the upper half-plane whose real part is a Brownian motion and whose imaginary part is an independent three-dimensional Bessel process starting at $0$. We also introduce, for $a>0$, $S_a = \inf\{s>0,\; y(R(u)-s) = a\}$. We have the following characterization.

\begin{Thm}
Let $z\in \mathbb{R}\backslash\{0\}$. For any $a>0$, under $\mu_z$, $(u(s))_{0\le s\le T_a}$ and $(u(R(u)-s)-z)_{0\le s\le S_a}$ are two independent $\Hb-$excursions stopped at the hitting time of $\{\Im(z)=a\}$. 
\end{Thm}

Through the change of measures $\mu_z$, $u$ therefore splits into two independent $\Hb-$excursions starting at $0$ and $z$ respectively.

\begin{proof}
The theorem follows from a similar application of the master formula. Let $f,g: U\rightarrow \R_+$ be two bounded continuous functions. Then 
\begin{align}
\n_+&\bigg(f(u(s), 0\le s\le T_a) g(u(R(u)-s), 0\le s\le S_a) {\mathcal M}_a \bigg) \label{beginning H excursion}\\
&= \n_+\bigg(f(u(s), 0\le s\le T_a)\mathds{1}_{\{T_a<\infty\}} \n_+ \bigg( g(u(R(u)-s), 0\le s\le S_a) {\mathcal M}_a \bigg| \F_{T_a}\bigg)\bigg). \label{H excursion}
\end{align}
 By the master formula,
\begin{align*}
\n_+ &\bigg( g(u(R(u)-s), 0\le s\le S_a) {\mathcal M}_a \bigg| \F_{T_a}\bigg) \\
&= \Eb\left[ \int_0^{T_{-a}} \mathrm{d}L_r \int_{-\infty}^{+\infty} \frac{\mathrm{d}x}{2\pi} \, \Eb\left[ g(x+x'+X_{T_{-a}-s},a+Y_{T_{-a}-s}, 0\le s\le S)\right]_{x'=X_r} \right],
\end{align*}
where $S:=\inf\{s>0,\; Y_{T_{-a}-s} = 0\}$. The change of variables $x+X_r \mapsto x$ provides 
\begin{align*}
& \n_+ \bigg( g(u(R(u)-s), 0\le s\le S_a) {\mathcal M}_a \bigg| \F_{T_a}\bigg) \\
&= \Eb\left[L_{T_{-a}}\right] \int_{-\infty}^{+\infty} \frac{\mathrm{d}x}{2\pi} \, \Eb\left[ g(x+X_{T_{-a}-s},a+Y_{T_{-a}-s}, 0\le s\le S)\right] \\
&= 2a \times \int_{-\infty}^{+\infty} \frac{\mathrm{d}x}{2\pi} \, \Eb\left[g(x+X_{T_{-a}-s},a+Y_{T_{-a}-s}, 0\le s\le S)\right].
\end{align*}
The path $(X_s,\, 0\le s\le T_{-a})$ is conditionally on $Y$ distributed as a linear Brownian motion stopped at time $T_{-a}$ (recall that $T_{-a}$ is a measurable function of $Y$). Since the Lebesgue measure is a reversible measure for the Brownian motion, by time-reversal, the  "law" of $(x+X_{T_{-a}-s},\, 0\le s \le S)$ for $x$ chosen with the Lebesgue measure is the "law" of a linear Brownian motion with initial measure the Lebesgue measure, stopped at time $S$ ($S$ is measurable with respect to $Y$).   Therefore, the integral $ \int_{-\infty}^{+\infty} \frac{\mathrm{d}x}{2\pi} \, \Eb\left[g(\ldots)\right]$ above is also 
\[
\int_{-\infty}^{+\infty} \frac{\mathrm{d}z}{2\pi} \, \Eb\left[g(z+X_s,a+Y_{T_{-a}-s}, 0\le s\le S)\right].
\]
Now we use that $(a+Y_{T_{-a}-s}, 0\le s\le S)$ has the law of a 3-dimensional Bessel process $V$ starting from $0$ and run until its hitting  time of $a$ (call this time $T_a^V$), see Corollary 4.6, Chap. VII of \cite{RY}. Hence the former integral is also
\[
  \int_{-\infty}^{+\infty} \frac{\mathrm{d}z}{2\pi} \, \Eb\left[ g(z+X_s,V_s,\, 0\le s \le T_a^V)\right].
\]
Plugging this into equation (\ref{H excursion}) triggers
\begin{align*}
\n_+&\bigg(f(u(s), 0\le s\le T_a) g(u(R(u)-s), 0\le s\le S_a) {\mathcal M}_a \bigg) \\
&=  2a \times \int_{-\infty}^{+\infty}
\frac{\mathrm{d}z}{2\pi} \, \Eb\left[ g(z+X_s,V_s,\, 0\le s \le T_a^V)\right]\n_+\bigg(f(u(s), 0\le s\le T_a) \mathds{1}_{\{T_a<\infty\}}\bigg)
 \\
&= \int_{-\infty}^{+\infty} \frac{\mathrm{d}z}{2\pi} \, \Eb\left[ g(z+X_s,V_s,\, 0\le s \le T_a^V)\right] \times \n_+\bigg(f(u(s), 0\le s\le T_a) \bigg| T_a<\infty\bigg).
\end{align*}
Moreover, using for example Williams' description of the It\^o's measure $n_+$ (Theorem 4.5,  Chap. XII, in \cite{RY}), the law of $(u(s), 0\le s\le T_a)$ conditionally on $\{T_a<\infty\}$ is the one of $(X_s,V_s,\, 0\le s \le T_a^V)$. We get eventually
\begin{align*}
\n_+&\bigg(f(u(s), 0\le s\le T_a) g(u(R(u)-s), 0\le s\le S_a) {\mathcal M}_a \bigg) \\
&= \int_{-\infty}^{+\infty} \frac{\mathrm{d}z}{2\pi} \, \Eb\left[ g(z+X_s,V_s,\, 0\le s \le T_a^V)\right] \times 
\Eb\left[ f(X_s,V_s,\, 0\le s \le T_a^V)\right].
\end{align*}

Finally, we disintegrate $\n_+$ over $x(R(u))$ to get 
\begin{align*}
    \int_{-\infty}^{+\infty}& \frac{\mathrm{d}z}{2\pi} \, \gamma_{z}\left[ f(u(s), 0\le s\le T_a) g(u(R(u)-s), 0\le s\le S_a) \frac{{\mathcal M}_a}{z^2}\right] \\
    &=\int_{-\infty}^{+\infty} \frac{\mathrm{d}z}{2\pi} \, 
    \Eb\left[ g(z+X_s,V_s,\, 0\le s \le T_a^V)\right] \times 
\Eb\left[ f(X_s,V_s,\, 0\le s \le T_a^V)\right]
    .
\end{align*}
Now multiply $g$ by any measurable function $h:\R\rightarrow \R_+$ of $x(R(u))$ to see that for Lebesgue-almost every $z\in \R$,
\begin{align} \label{eq:change measure 1}
    \gamma_{z}&\left[ f(u(s), 0\le s\le T_a) g(u(R(u)-s), 0\le s\le S_a) \frac{{\mathcal M}_a}{z^2}\right] \\
    &= \Eb\left[ g(z+X_s,V_s,\, 0\le s \le T_a^V)\right] \times 
\Eb\left[ f(X_s,V_s,\, 0\le s \le T_a^V)\right]. \label{eq:change measure 2}
\end{align}
The right-hand side of this equation is a continuous function of $z$. Moreover, by scaling (Lemma \ref{l:scaling}), for $z>0$ the left-hand term can be written
\begin{align} 
\gamma_{z}&\left[ f(u(s), 0\le s\le T_{a}) g(u(R(u)-s), 0\le s\le S_a) \frac{{\mathcal M}_a}{z^2}\right]  \\
&=\gamma_1 \left[ f(zu(s/z^2), 0\le s\le T_{a/z} z^2) g(zu(R(u)-\frac{s}{z^2}), 0\le s\le S_{a/z} z^2) {\mathcal M}_{a/z}\right]. \label{eq:limit}
\end{align}
Since equality (\ref{eq:change measure 1})-(\ref{eq:change measure 2}) holds almost everywhere, it must hold for a dense set of $z>0$. Take $z\searrow 1$ along such a subsequence. By Lemma \ref{lem:continuity} and the observation that $T_{a/z}\rightarrow T_a$, $S_{a/z}\rightarrow S_a$, we get the convergences $(zu(s/z^2), 0\le s\le T_{a/z} z^2) \rightarrow (u(s), 0\le s\le T_a)$ and $(zu(R(u)-\frac{s}{z^2}), 0\le s\le S_{a/z} z^2) \rightarrow (u(R(u)-s), 0\le s\le S_a)$ in $U$. In addition, we know that $\mathcal{M}_{a/z}\rightarrow \mathcal{M}_a$ almost surely and $\gamma_1 \left[{\mathcal M}_{a/z}\right] \rightarrow \gamma_1\left[{\mathcal M}_a\right]$ (both these expressions are equal to $1$ by Proposition \ref{martingale}). By Scheffé's lemma, ${\mathcal M}_{a/z} \underset{z\searrow 1}{\longrightarrow} {\mathcal M}_a$ in $L^1$. When $z\searrow 1$, this turns (\ref{eq:limit}) into
\begin{align*} 
    \gamma_{z}&\left[ f(u(s), 0\le s\le T_a) g(u(R(u)-s), 0\le s\le S_a) \frac{{\mathcal M}_a}{z^2}\right] \\
    &\underset{z\searrow 1}{\longrightarrow} \gamma_1 \left[ f(u(s), 0\le s\le T_a) g(u(R(u)-s), 0\le s\le S_a) {\mathcal M}_a\right]. 
\end{align*}
Therefore (\ref{eq:change measure 1})-(\ref{eq:change measure 2}) holds for $z=1$, and then for any $z$ by scaling. So we proved that for all $z\in \mathbb{R} \backslash\{0\}$,
\begin{align*} 
    \gamma_z &\left[ f(u(s), 0\le s\le T_a) g(u(R(u)-s), 0\le s\le S_a) \frac{{\mathcal M}_a}{z^2}\right] \\
    &= \Eb\left[ g(z+X_s,V_s,\, 0\le s \le T_a^V)\right] \times \Eb\left[ f(X_s,V_s,\, 0\le s \le T_a^V)\right] ,
\end{align*}
which is simply 
\begin{align}
    \mu_{z}&\left[f(u(s), 0\le s\le T_a) g(u(R(u)-s), 0\le s\le S_a) \right] \notag \\
    &= \Eb\left[ g( z+X_s,V_s,\, 0\le s \le T_a^V)\right] \times \Eb\left[ f(X_s,V_s,\, 0\le s \le T_a^V)\right] . \label{H excursion formula}
\end{align}
This proves that under $\mu_z$, the processes $(u(s))_{0\le s\le T_a}$ and $(u(R(u)-s)-z)_{0\le s\le S_a}$ are independent $\Hb-$excursions stopped at the hitting time of $\{\Im(z) = a\}$.  
\end{proof}

\begin{Rk}
This gives a new insight on why the Cauchy process should be hidden in some sense in the law of $\Xi$: under the tilted measure, $u$ splits into two independent $\Hb$--excursions and so the size at some level $a$ of the spine going to infinity is just the difference of two Brownian motions started from infinity taken at their hitting time of $\{\Im(z) = a\}$.
\end{Rk}


\section{The growth-fragmentation process of excursions in $\Hb$}
\label{s:cellsystem}

In this section, we summarize the previous results in the language of the self-similar growth-fragmentations introduced by Bertoin in \cite{B}. The main reference here is \cite{BBCK}, but for the sake of completeness we shall recall in the first paragraph the bulk of the construction of such processes. At the heart of this section lies the calculation of the cumulant function. We recover the cumulant function of \cite{BBCK}, formula (19), in the specific case when $\theta=1$. Recall the definition of $Z$ in Theorem \ref{thm:locally largest}. The process $Z$ starting at $z<0$ is defined to be the negative of the process $Z$ starting at $-z$.



\subsection{Construction of $\overline{\mathbf{X}}$}

We explain how one can define the \emph{cell system} driven by $Z$. We use the Ulam tree $\Ub = \cup_{i=0}^{\infty} \N^i$, where $\N=\{1,2,\ldots\}$, to encode the genealogy of the cells (we write $\N^0 = \{\varnothing\}$, and $\varnothing$ is called the Eve cell). A node $u\in \Ub$ is a list $(u_1,\ldots, u_i)$ of positive integers where $|u|=i$ is the \emph{generation} of $u$. The children of $u$ are the lists in $\N^{i+1}$ of the form $(u_1,\ldots,u_i,k)$, with $k\in \N$. A \emph{cell system} is a family $\Xcal = (\Xcal_u, u\in\Ub)$ indexed by $\Ub$, where $\Xcal_u = (\Xcal_u(a))_{a\ge 0}$ is meant to describe the evolution of the size or mass of the cell $u$ with its age $a$.

To define the cell system driven by $Z$, we first define $\Xcal_{\varnothing}$ as $Z$, started from some initial mass $z\neq 0$, and set $b_{\varnothing}=0$. Observe the realization of $\Xcal_{\varnothing}$ and its jumps. Since $Z$ hits $0$ in finite time, we may rank the sequence of jump sizes and times $(x_1,\beta_1), (x_2,\beta_2),\ldots$ of $-\Xcal_{\varnothing}$ by decreasing order of the $|x_i|$'s. Conditionally on these jump sizes and times, we define the first generation of our cell system $\Xcal_i, i\in\N,$ to be independent with $\Xcal_i$ distributed at $Z$, starting from $x_i$. We also set $b_i = b_{\varnothing}+\beta_i$ for the birth time of the particle $i\in\N$. By recursion, one defines the law of the $n$-th generation given generations $1,\ldots,n-1$ in the same way. Hence the cell labelled by $u=(u_1,\ldots, u_n)\in\N^n$ is born from $u'=(u_1,\ldots, u_{n-1})\in\N^{n-1}$ at time $b_u=b_{u'}+\beta_{u_n}$, where $\beta_{u_n}$ is the time of the $u_n$-th largest jump of $\Xcal_{u'}$, and conditionally on $\Xcal_{u'}(\beta_{u_n})-\Xcal_{u'}(\beta_{u_n}^-)=-y$, $\Xcal_u$ has the law of $Z$ with initial value $y$ and is independent of the other daughter cells at generation $n$. We write $\zeta_u$ for the lifetime of the particle $u$. We may then define, for $a\ge 0$,
\begin{equation}\label{def:Xbar}
\overline{\mathbf{X}}(a):= ( \Xcal_u(a-b_u), \; u\in \Ub \; \text{and} \; b_u\le a <b_u+\zeta_u  ) ,
\end{equation}
as the family of the sizes of all the cells alive at time $a$.  We arrange the elements in $\overline{\mathbf{X}}(a)$  in descending order of their absolute values.


\subsection{The growth-fragmentation process of excursions in $\Hb$}

We restate Theorem \ref{thm:main}. Beware that the signed growth-fragmentation $\overline{\mathbf{X}}$ in this section starts from $z$.
\begin{Thm} \label{thm:growth-frag}
Let $z\in \mathbb{R}\backslash\{0\}$. Under $\gamma_z$,
\[(\overline{\mathbf{X}}(a), a\ge 0) \overset{\text{law}}{=} \left( (\Delta e^{a,+}_i, \; i\ge 1  ), a\ge 0\right).\]
\end{Thm}
\begin{proof} 
Let $u\in U^+$ be such that the locally largest excursion described in Subsection \ref{Sec:locally-largest} is well-defined, \emph{i.e.} $u$ has no loop above any level, has distinct local minima, and no splitting in two equal sizes (this set of excursions has full probability under $\gamma_z$). This gives our Eve cell process. The independence of the daughter excursions given their size at birth has already been proved in Theorem \ref{idp}, and we have taken $Z$ according to the law of the largest fragment in Theorem \ref{thm:locally largest}, so it remains to prove that every excursion can be found in the genealogy of $\overline{\mathbf{X}}$ as constructed in the former section.

For $a\ge 0$, we denote by $\overline{\mathbf{X}}^{exc}(a)$ the set of all excursions associated to the sizes in $\overline{\mathbf{X}}(a)$. Let $0\le t\le R(u)$ such that $\Im(u(t))>a$. We want to show that $e_a^{(t)}\in\overline{\mathbf{X}}^{exc}(a)$. Set
\[ \mathcal{A} = \left\{a'\in[0,a], \; e_{a'}^{(t)} \in \overline{\mathbf{X}}^{exc}(a')\right\}.\]
Then $\mathcal{A}$ is an interval containing $0$. 
\begin{itemize}
    \item[$\bullet$] $\mathcal{A}$ is open in $[0,a]$. Let $a'\in\mathcal{A}$ with $a'<a$. Write $e_{b}^{(\tau^{\bullet})}, b\ge a',$ for the locally largest excursion inside $e_{a'}^{(t)}$. Then for small enough $\varepsilon>0$, $e_{a'+\varepsilon}^{(t)} = e_{a'+\varepsilon}^{(\tau^{\bullet})}$. Indeed, the first height $b\ge a'$ when $e_{b}^{(t)} \ne e_{b}^{(\tau^{\bullet})}$ is equal to  the minimum of $y(s)$ for $s$ between $t$ and $\tau^{\bullet}$, and so it is stricly above $a'$. This implies that $a'+\varepsilon \in\mathcal{A}$ since $e_{a'+\varepsilon}^{(\tau^{\bullet})}\in\overline{\mathbf{X}}^{exc}(a'+\varepsilon)$ as the locally largest excursions are in the genealogy.
    \item $\mathcal{A}$ is closed in $[0,a]$. Let $a_n$ be a sequence of elements in $\mathcal{A}$ increasing to $a_{\infty}$. For all $\varepsilon>0$, there exists $\delta>0$ such that:
    \[ \forall a'\in(a_{\infty}-\delta,a_{\infty}), \quad |\Delta e_{a'}^{(t)} - \Delta e_{a_{\infty}^-}^{(t)}| < \varepsilon.\]
    Then for all $a_1,a_2\in(a_{\infty}-\delta,a_{\infty})$,
    \[ |\Delta e_{a_1}^{(t)} - \Delta e_{a_2}^{(t)}| \le |\Delta e_{a_1}^{(t)} - \Delta e_{a_{\infty}^-}^{(t)}| + |\Delta e_{a_2}^{(t)} - \Delta e_{a_{\infty}^-}^{(t)}| < 2 \varepsilon.\]
    Take $\varepsilon=|\Delta e_{a_{\infty}^-}^{(t)}|/4$ and $N$ large enough so that $a_N\in (a_{\infty}-\delta,a_{\infty})$. Then the excursion $e_{a_N}^{(t)}$ is such that for all $a'\in[a_N,a_{\infty})$, $e_{a'}^{(t)}$ is taken along the locally largest excursion inside $e_{a_N}^{(t)}$. Indeed, it follows from these inequalities that for all $a_1,a_2 \in(a_{\infty}-\delta,a_{\infty})$, $|\Delta e_{a_1}^{(t)} - \Delta e_{a_2}^{(t)}|\le\frac12 |\Delta e_{a_{\infty}^-}^{(t)}|< |\Delta e_{a_1}^{(t)}|$ , then take $a_1=a'$ and $a_2 \nearrow a'$. This entails that $a_{\infty}\in\mathcal{A}$.
\end{itemize}
By connectedness $\mathcal{A}$ must be $[0,a]$. This concludes the proof.
\end{proof}


\subsection{The cumulant function}

The process $\overline{\mathbf{X}}$ is not a growth-fragmentation in the sense of \cite{BBCK} because it carries negative masses. We show in this section that if one discards all cells with negative masses together with their progeny, one obtains one of the growth-fragmentation processes studied in \cite{BBCK}. \\

Formally, let $\mathbf{X}$ defined by \eqref{def:Xbar} where we only consider  the $u$'s such that $\Xcal_v(b_v) >0$ for all ancestors $v$ of $u$ (including itself) in the Ulam tree.  The process $\mathbf{X}$ is a growth-fragmentation in the sense of \cite{BBCK}. It is characterized by its self-similarity index $\alpha=-1$ and its cumulant function defined for $q\ge 0$, by
\[\kappa(q) := \Psi(q) + \int_{-\infty}^0 (1-\mathrm{e}^y)^q\Lambda(\mathrm{d}y),\]
where $\Lambda$ denotes the Lévy measure of the Lévy process $\xi$.

The following proposition is Proposition 5.2 of \cite{BBCK} in the case $\theta=1$, $\hat{\beta} = 1$ and $\gamma = \hat{\gamma} = 1/2$ with the additional factor $2$ (corresponding to a time change). 
\begin{Prop}\label{p:cumulant}
Let $\omega_+ =\omega=  5/2$, and $\Phi^+(q) = \kappa(q+\omega_+)$ for $q\ge 0$. Then $\Phi^+$ is the Laplace exponent of a symmetric Cauchy process conditioned to stay positive, namely
\begin{equation}
    \Phi^+(q) = -2\, \frac{\Gamma(\frac12-q)\Gamma(\frac32+q)}{\Gamma(-q)\Gamma(1+q)}, \quad -\frac32 < q < \frac12.
\end{equation}

Furthermore, the associated growth-fragmentation $X$ has no killing and its cumulant function is 
\begin{equation} \label{cumulant}
    \kappa(q) = -2 \, \frac{\cos(\pi q)}{\pi} \Gamma(q-1)\Gamma(3-q), \quad 1<q<3.
\end{equation}
\end{Prop}

\begin{Rk}
In \cite{BBCK}, the roots of $\kappa$ pave the way to remarkable martingales. It should not come as a surprise that in our case these roots happen to be $\omega_-=\frac32$ and $\omega_+=\frac52$. Indeed, the $h$-transform for the symmetric Cauchy process conditioned to stay positive (resp. conditioned to hit $0$ continuously) is given by $x\mapsto x^{1/2}$ (resp. $x\mapsto x^{-1/2}$). This turns the martingale in Proposition \ref{prop:martingale} into the sum over all masses in $\mathbf{X}$ to the power $\omega_+=2+\frac12$, and $\omega_- = 2-\frac12$ respectively, which are exactly the quantities considered in \cite{BBCK}.
\end{Rk}

\begin{proof}
The strategy is as follows. In view of Theorem 5.1 in \cite{BBCK}, we first compute $\kappa(q+\omega)-\kappa(\omega)$ and we put it in a L\'evy-Khintchin form so as to retrieve the Laplace exponent of the L\'evy process involved in the Lamperti representation of a Cauchy process conditioned to stay positive, which is known from \cite{CC}. We then show that $\kappa(\omega)=0$, and therefore deduce the expression of $\kappa$. 

Recall first that by definition 
\[ \kappa(q) = \Psi(q) + \int_{-\infty}^0 (1-\mathrm{e}^y)^q\Lambda(\mathrm{d}y),\]
with $\Psi$ given by (\ref{Lapl loc largest}). In fact, we rather use formula (\ref{Lapl indicator}), which is closer to \cite{CC}:
\[
\Psi(q) = -\frac2\pi \left(\ln(2)+\frac32\right)q + \frac2\pi \int_{y>-\ln(2)} \left(\mathrm{e}^{q y}-1-q(\mathrm{e}^y-1)\mathds{1}_{|\mathrm{e}^y-1|<1}\right) \frac{\mathrm{e}^{-y}\mathrm{d}y}{(\mathrm{e}^y-1)^2}. 
\]
Let $-\frac32<q<\frac12$. Then 
\begin{align*}
\frac\pi2&(\kappa(q+\omega)-\kappa(\omega)) \\
&=-\left(\ln(2)+\frac32\right)q+\int_{y>-\ln(2)} \left(\mathrm{e}^{(q+\omega) y}-\mathrm{e}^{\omega y}-q(\mathrm{e}^y-1)\mathds{1}_{|\mathrm{e}^y-1|<1}\right) \frac{\mathrm{e}^{-y}\mathrm{d}y}{(\mathrm{e}^y-1)^2} \\
& \qquad + \int_{-\ln(2)}^0 \left((1-\mathrm{e}^y)^{q+\omega}-(1-\mathrm{e}^y)^{\omega}\right)\frac{\mathrm{e}^{-y}\mathrm{d}y}{(\mathrm{e}^y-1)^2}.
\end{align*}
Performing the change of variables $e^x=1-e^y$ in the second integral entails
\begin{align*}
\frac\pi2&(\kappa(q+\omega)-\kappa(\omega)) \\
&= -\left(\ln(2)+\frac32\right)q+\int_{y>-\ln(2)} \left(\mathrm{e}^{(q+\omega) y}-\mathrm{e}^{\omega y}-q(\mathrm{e}^y-1)\mathds{1}_{|\mathrm{e}^y-1|<1}\right) \frac{\mathrm{e}^{-y}\mathrm{d}y}{(\mathrm{e}^y-1)^2} \\
& \qquad + \int_{-\infty}^{-\ln(2)} \left(\mathrm{e}^{(q+\omega)x}-\mathrm{e}^{\omega x}\right)\frac{\mathrm{e}^{-x}\mathrm{d}x}{(\mathrm{e}^x-1)^2} \\
&= -\left(\ln(2)+\frac32\right)q+\int_{-\infty}^{+\infty} \left(\mathrm{e}^{(q+\omega) y}-\mathrm{e}^{\omega y}-q\mathrm{e}^{\omega y}(\mathrm{e}^y-1)\mathds{1}_{|\mathrm{e}^y-1|<1}\right) \frac{\mathrm{e}^{-y}\mathrm{d}y}{(\mathrm{e}^y-1)^2} \\
& \qquad + q \int_{y>-\ln(2)} \left( \mathrm{e}^{\omega y}(\mathrm{e}^y-1)-(\mathrm{e}^y-1)\right)\mathds{1}_{|\mathrm{e}^y-1|<1}\frac{\mathrm{e}^{-y}\mathrm{d}y}{(\mathrm{e}^y-1)^2} \\
& \qquad + q \int_{-\infty}^{-\ln(2)}\mathrm{e}^{\omega y}(\mathrm{e}^y-1)\underbrace{\mathds{1}_{|\mathrm{e}^y-1|<1}}_{=1}\frac{\mathrm{e}^{-y}\mathrm{d}y}{(\mathrm{e}^y-1)^2} \\
&= -\left(\ln(2)+\frac32\right)q+\int_{-\infty}^{+\infty} \left(\mathrm{e}^{q y}-1-q(\mathrm{e}^y-1)\mathds{1}_{|\mathrm{e}^y-1|<1}\right) \frac{\mathrm{e}^{(\omega-1)y}\mathrm{d}y}{(\mathrm{e}^y-1)^2} \\
& \qquad +  q \int_{-\ln(2)}^{\ln(2)} \left( \mathrm{e}^{\omega y}-1\right)\frac{\mathrm{e}^{-y}\mathrm{d}y}{\mathrm{e}^y-1} + q \int_{-\infty}^{-\ln(2)}\mathrm{e}^{\omega y}\frac{\mathrm{e}^{-y}\mathrm{d}y}{\mathrm{e}^y-1}.
\end{align*}
Because $\omega=5/2$, this has the form of $\Phi^{\uparrow}$ of Corollary 2 in \cite{CC} for the symmetric Cauchy process ($\alpha=1$ and $\rho=1/2$), apart from a possible extra drift. We now show that the drifts do in fact coincide. Let $I$ and $J$ denote the last two integrals in the above expression. Using the change of variables $x=e^y$, we get
\begin{align*}
    I &= \int_{1/2}^2 \frac{x^{5/2}-1}{x^2(x-1)} \mathrm{d}x, \\
    J &= \int_0^{1/2} \frac{\sqrt{x}}{x-1} \mathrm{d}x.
\end{align*}
Now 
\[I= \int_{1/2}^2 \frac{x^{5/2}-x^2}{x^2(x-1)} \mathrm{d}x + \int_{1/2}^2 \frac{x^2-1}{x^2(x-1)} \mathrm{d}x = \int_{1/2}^2 \frac{\sqrt{x}-1}{x-1} \mathrm{d}x + \underbrace{\int_{1/2}^2 \frac{x+1}{x^2} \mathrm{d}x}_{:=I_1}, \] 
and
\[J= \int_0^{1/2} \frac{\sqrt{x}-1}{x-1} \mathrm{d}x + \underbrace{\int_0^{1/2} \frac{1}{x-1} \mathrm{d}x}_{:=J_1}.\]
One can check that $I_1 + J_1 = \ln(2)+\frac32$. Therefore the linear term in the above expression of $\kappa(q+\omega)-\kappa(\omega)$ is precisely 
\[a_+ = \frac{2}{\pi}\int_0^{2} \frac{\sqrt{x}-1}{x-1} \mathrm{d}x = \frac{2}{\pi}\int_0^{1} \frac{\sqrt{1+u}-1}{u} \mathrm{d}u - \frac{2}{\pi}\int_0^{1} \frac{\sqrt{1-u}-1}{u} \mathrm{d}u,\]
which is exactly $a_+ = 2a^{\uparrow}$ as defined in Corollary 2 of \cite{CC} for the symmetric Cauchy process. Note that there is a sign error in formula (17) of the latter paper. Hence Corollary 2 of \cite{CC} triggers that $\kappa(q+\omega)-\kappa(\omega)$ is twice the Laplace exponent of a Cauchy process conditioned to stay positive, and now by \cite{KP}, we deduce
\[\kappa(q+\omega)-\kappa(\omega) = -2\, \frac{\Gamma(\frac12-q)\Gamma(\frac32+q)}{\Gamma(-q)\Gamma(1+q)}, \quad -\frac32 < q < \frac12.\]

Taking $q=-1/2$ in this formula, one sees that $\kappa(2)-\kappa(5/2) = -\frac2\pi$. Yet one can easily compute $\kappa(2)$ from the definition of $\kappa$. Simple calculations left to the reader actually lead to $\kappa(2) = -\frac2\pi$, and thus $\kappa(5/2)=0$. Finally, we recovered the expression of $\Phi^+$, and using Euler's reflection formula 
\[\kappa(q) = -2 \, \frac{\cos(\pi q)}{\pi} \Gamma(q-1)\Gamma(3-q), \quad 1<q<3.\]
\end{proof}

%
%
\section{Convergence of the derivative martingale} \label{s:deriv}


Recall the construction of the cell system in Section \ref{s:cellsystem} and that for $u\in \Ub$, $|u|$ denotes its generation. The collection $(\ln(|\Xcal_u(0)|),\, u\in \Ub)$ defines a branching random walk, see \cite{ZShi} for a general reference on branching random walks. We will work under the associated filtration 
\[
\Gscr_n := \sigma\left(\Xcal_u, \; |u|\le n\right), \quad n\ge 0.
\]

\noindent By construction, with the notation of Theorem \ref{thm:locally largest}, one has for all suitable measurable function $f$ such that $f(0)=0$, under $\gamma_z$,
$$
\sum_{|u|=1} f(\Xcal_u(0))
=
\sum_{a\ge 0} f\left(z{\mathrm e}^{\xi(a)}- z{\mathrm e}^{\xi(a-)} \right).
$$

\noindent  From there, one can check by computations (making use of the expression of the cumulant function found in \eqref{cumulant}) that
$$
\gamma_z\left[\sum_{|u|=1} |\Xcal_u(0)|^2  \right]=z^2, \qquad \gamma_z\left[\sum_{|u|=1} |\Xcal_u(0)|^2 \ln(|\Xcal_u(0)|) \right]=0 
$$

\noindent which implies that the martingale $\Mcal(n):=\sum_{|u|=n+1} |\Xcal_u(0)|^2$ is the critical martingale for the branching random walk. In this case, $\Mcal(n)$ converges to $0$. In order to have a non-trivial limit, one needs to consider the so-called derivative martingale defined by 
\[
\Dcal(n) := - \sum_{|u|=n+1} \ln(|\Xcal_u(0)|) |\Xcal_u(0)|^2, \quad n\ge 0.
\]

\noindent The aim of this section is to show that this limit is twice the duration of the excursion. 

First notice that the duration of the excursion is measurable with respect to the cell system, or equivalently to the growth-fragmentation $\overline{\mathbf{X}}$. Indeed, the number of excursions above level $a$  with height greater than $\varepsilon>0$ is measurable with respect to  $\overline{\mathbf{X}}$, hence the local time of the excursion at level $a$ is also measurable, and so is the total duration of the excursion by the occupation times formula. Then, by L\'evy's martingale convergence theorem, one  would only have to show that the derivative martingale is the conditional expectation of the duration with respect to the  filtration $\Gscr_n $. But the duration is not integrable, hence this strategy cannot work. Instead, one needs to use a truncation procedure  introduced in \cite{BiKy}. 

Let $C>0$ and denote by $\Ub^{(C)}$ the set of labels obtained from $\Ub$ by killing all the cells (with their progeny) when their size is larger than $C$ in absolute value.

\begin{Lem} \label{lem:time spent C}
Let $z\ne 0$ and
\[ T_C:= \int_0^{R(u)}\mathds{1}_{\{\forall 0\le b\le y(t), \, |\Delta e_b^{(t)}| < C\}} \mathrm{d}t,\]
be the amount of time spent by excursions with size between $-C$ and $C$. Then
\[
\gamma_z(T_C) 
=
\pi z^2 \Rcal_C(z/2)  \mathds{1}_{\{|z| < C\}}
\]

\noindent where $\Rcal_C(z):= -\frac{1}{2\pi}\left( \ln \big|\sqrt{(1+{\tilde z}) / (1-{\tilde z}}) -1 \big| - \ln \big|\sqrt{(1+\tilde z ) / (1-\tilde z)} +1 \big|\right)$, with $\tilde z={2 z\over C}$, is the Green function at $0$ of the Cauchy process in $(-{C\over 2},{C\over 2})$, see \cite{daviaud}.  
\end{Lem}

\begin{proof}
Lemma \ref{lem:time spent C} follows from an application of Bismut's description of $\n_+$. Indeed, if $f:\R \rightarrow \R_+$ is a nonnegative measurable function, then by Proposition \ref{Bismut},
\begin{align*}
\n_+&\left(\int_0^{R(u)}\mathds{1}_{\{\forall 0\le b\le y(t), \, |\Delta e_b^{(t)}| < C\}} f(x(R(u)) \mathrm{d}t\right) 
\\
&= \int_0^{\infty} \mathrm{d}a \,\Eb\left[ \mathds{1}_{\{\forall 0\le b\le a, \, |X'(T'_{-b})-X(T_{-b})| < C\}} f(X'(T'_{-a})-X(T_{-a}))\right]
\end{align*}

\noindent where under $\Pb$, $(X,Y)$ and $(X',Y')$ are independent planar Brownian motions starting from $0$, and $T_{-a}$ and $T'_{-a}$ denote their respective hitting times of $\{\Im(z)=-a\}$. Observe that $a\mapsto X'(T'_{-a})-X(T_{-a})$ is the double of a Cauchy process (see for instance Proposition 3.11 of \cite{RY}, Chap. III). By definition of the Green's function, we get
\[
\n_+\left( \int_0^{R(u)}\mathds{1}_{\{\forall 0\le b\le y(t), \, |\Delta e_b^{(t)}| < C\}} f(x(R(u)) \mathrm{d}t \right)
= \int_{-C/2}^{C/2} f(2z)\Rcal_C(z)\mathrm{d} z.
\]

\noindent We deduce that
\[
\n_+\left(T_C f(x(R(u))\right) 
=  
\frac12 \int_{-C}^{C} f(z)\Rcal_C(z/2)\mathrm{d} z
=  \int_{-C}^{C} \frac{f(z)}{2 \pi z^2} \pi z^2 \Rcal_C(z/2)\mathrm{d} z.
\]

\noindent Therefore, by disintegration over $x(R(u))$, for Lebesgue--almost every $z$, we get
\[
\gamma_z\left(T_C\right) 
=
\pi z^2 \Rcal_C(z/2)\mathds{1}_{\{|z|<C\}}.
\]
The fact that it holds for all $z$ is obtained through the usual continuity arguments.
\end{proof}

\bigskip

Let 
\[
\Dcal^{(C)}(n) := \pi \sum_{u \in \Ub^{(C)}: \, |u|=n+1} \Rcal_C(\Xcal_u(0)/2) |\Xcal_u(0)|^2, \quad n\ge0.
\]
\begin{Cor} \label{cor:truncated derivative}
The following identity holds for all $n\ge 0$ and $z\ne 0$:
\[
\gamma_z\left(T_C \, \Big| \, \Gscr_n\right) 
=
 \Dcal^{(C)}(n).
\]
Consequently, $(\Dcal^{(C)}(n), n\ge 0)$ is a uniformly integrable $(\Gscr_n)_{n\ge 0}$--martingale.
\end{Cor}
\begin{proof}
For simplicity, we prove this for $n=0$: the identity then follows by induction from the branching property. Splitting the integral over the children of $\Xi$ and using Theorem \ref{idp}, we have for $|z|<C$,
\[
\gamma_z\left(\int_0^{R(u)}\mathds{1}_{\{\forall 0\le b\le y(t), \, |\Delta e_b^{(t)}| \le C\}} \mathrm{d}t \, \Big| \, \Gscr_1\right) 
=
\sum_{i\ge 1} \gamma_{\Delta e_i}\left(\int_0^{R(u)}\mathds{1}_{\{\forall 0\le b\le y(t), \, |\Delta e_b^{(t)}| \le C\}} \mathrm{d}t\right), 
\]
where the $e_i, i\ge 1$ denote the excursions created by the jumps of $\Xi$. Now, using Lemma \ref{lem:time spent C}, we immediately get
\[
\gamma_z\left(\int_0^{R(u)}\mathds{1}_{\{\forall 0\le b\le y(t), \, |\Delta e_b^{(t)}| \le C\}} \mathrm{d}t \, \Big| \, \Gscr_1 \right) 
=
\sum_{i\ge 1} \Rcal_C(\Delta e_i/2) \pi |\Delta e_i|^2\mathds{1}_{\{|\Delta e_i|< C\}},
\]
which is the desired equality.
\end{proof}

\bigskip

\begin{Thm}
Under $\gamma_z$, the derivative martingale $(\Dcal(n), n \ge 0)$ converges almost surely towards twice the duration $R(u)$ of the Brownian excursion.
\end{Thm}

\begin{proof}
The proof is standard in the branching random walk literature, see \cite{BiKy} or \cite{ZShi}. It suffices to prove it on the event where all excursions have size smaller than $C$, this for any $C > 0$. Let then $C > 0$ and suppose the corresponding event holds. Using that $\Rcal_C(z) = -{1\over 2\pi} \ln(z) + O_z(1)$ when $z\to 0$ and that the martingale $\Mcal(n)$ converges to $0$,  we get that 
$$
\Dcal^{(C)}(n) \underset{n\to\infty}{\sim} 
\frac12 \Dcal(n). 
$$

\noindent Since on our event, $T_C=R(u)$, L\'evy's martingale convergence theorem together with Corollary \ref{cor:truncated derivative} imply the result. 
\end{proof}





\end{document}